\documentclass[10pt,oneside,leqno]{amsart}
\usepackage{amsxtra}
\usepackage{amsopn}
\usepackage{color}
\usepackage{amsmath,amsthm,amssymb}
\usepackage{amscd}
\usepackage{amsfonts}
\usepackage{latexsym}
\usepackage{verbatim}
\usepackage{graphicx}
\usepackage{graphics}
\usepackage{multirow}
\usepackage{lscape}
\usepackage{nicefrac}
\usepackage{mathabx}
\usepackage[all]{xy}
\usepackage{float}
\usepackage[table]{xcolor}
\usepackage{colortbl}
\usepackage{hhline}

\usepackage[hypertexnames=false,
 backref=UseNone,
    pdftex,
    pdfpagemode=UseNone,
    breaklinks=true,
    extension=pdf,
    colorlinks=true,
    linkcolor=blue,
    citecolor=blue,
    urlcolor=blue,
]
{hyperref}

\theoremstyle{plain}
\newtheorem{theorem}{Theorem}[section]
\newtheorem{definition}[theorem]{Definition}
\newtheorem{lemma}[theorem]{Lemma}
\newtheorem{proposition}[theorem]{Proposition}
\newtheorem{corollary}[theorem]{Corollary}
\newtheorem{remark}[theorem]{Remark}
\newtheorem{notation}[theorem]{Notation}

\newtheorem{remark-question}[section]{Remark-Question}

\newcommand\C{{\mathbb C}}

\newcommand\fra{{\mathfrak a}} 

\newcommand\frf{{\mathfrak f}}
\newcommand\frg{{\mathfrak g}}

\newcommand\gc{{\frg_{\mathbb C}}}
\newcommand\Real{{\mathfrak R}{\frak e}\,} 
\newcommand\Imag{{\mathfrak I}{\frak m}\,}

\newcommand\db{{\bar{\partial}}}

\DeclareMathOperator{\sign}{sign}

\definecolor{fondo}{rgb}{0.93,0.93,0.93}

\definecolor{m}{rgb}{0.9,0,0.9}

\makeatletter
\renewcommand*{\eqref}[1]{%
  \hyperref[{#1}]{\textup{\tagform@{\ref*{#1}}}}%
}
\makeatother

\begin{document}
\title[]{Complex structures on nilpotent Lie algebras\\ with one-dimensional center}

\author{Adela Latorre}
\address[A. Latorre]{Departamento de Matem\'atica Aplicada,
Universidad Polit\'ecnica de Madrid,
Avda. Juan de Herrera 4,
28040 Madrid, Spain}
\email{adela.latorre@upm.es}

\author{Luis Ugarte}
\address[L. Ugarte]{Departamento de Matem\'aticas\,-\,I.U.M.A.\\
Universidad de Zaragoza\\
Campus Plaza San Francisco\\
50009 Zaragoza, Spain}
\email{ugarte@unizar.es}

\author{Raquel Villacampa}
\address[R. Villacampa]{Centro Universitario de la Defensa\,-\,I.U.M.A., Academia General
Mili\-tar, Crta. de Huesca s/n. 50090 Zaragoza, Spain}
\email{raquelvg@unizar.es}


\begin{abstract}
We classify the nilpotent Lie algebras of real dimension eight and minimal center that admit a complex structure. Furthermore, for every such nilpotent Lie algebra $\frg$, we describe the space of complex structures on $\frg$ up to isomorphism.
As an application, the nilpotent Lie algebras 
having a 
non-trivial abelian $J$-invariant ideal are classified up to eight dimensions.
\end{abstract}

\maketitle

\setcounter{tocdepth}{2}      

\section{Introduction}

In the last decades, the study of complex nilmanifolds has proved to be very useful in the understanding of several 
aspects of compact complex manifolds.
By a complex nilmanifold we mean a compact quotient of a simply connected nilpotent Lie group~$G$ endowed with
a \emph{left invariant} complex structure $J$ (namely, defined on the Lie algebra $\frg$ of $G$) by a cocompact discrete subgroup $\Gamma$. Hence, the study of nilpotent Lie algebras $\frg$ with complex structures~$J$ constitutes a crucial step in the construction of complex nilmanifolds.

Some recent results where complex nilmanifolds play an important role can be found in 
\cite{Ang-libro}--\cite{BR}, \cite{CF}, \cite{CFP}, \cite{DF}--\cite{FV-correction}, \cite{MPPS}, \cite{OOS}--\cite{RTW} and the references therein. These include cohomological aspects of compact complex manifolds (Dolbeault, Bott-Chern, Aeppli cohomologies and Fr\"olicher spectral sequence), existence of different classes of Hermitian metrics (as SKT, locally conformal K\"ahler or balanced metrics, among others),
as well as the
behaviour 
of complex properties under holomorphic deformations.
 It is worth to note that in the previous results the complex structures $J$ on the Lie algebras $\frg$ underlying the nilmanifolds usually satisfy $Z(\frg)\cap J\left(Z(\frg)\right)\neq\{0\}$, being $Z(\frg)$ the center of $\frg$. These complex structures are known as \emph{quasi-nilpotent} (see  Definition~\ref{tipos_J}) and include complex-parallelizable as well as abelian complex structures.
By \cite[Section 2]{LUV-SnN}, any $\frg$ with a quasi-nilpotent complex structure can be constructed as a certain extension of a lower dimensional nilpotent Lie algebra endowed with a complex structure. Therefore, the `essentially new' complex structures $J$ that arise in each
even real dimension are those for which
the only $J$-invariant subspace of the  center $Z(\frg)$ is the trivial one.
These are called \emph{strongly non-nilpotent}, or \emph{SnN} for short, complex structures.

Although the complex geometry of nilmanifolds endowed with SnN complex structures still remains to be studied in general, some partial results have been obtained. 
For instance, in \cite{LUV-complex,LUV-gen-complex} several families of such complex nilmanifolds are constructed, 
allowing to prove the
existence of infinitely many real homotopy types of nilmanifolds with 
balanced
Hermitian metrics 
and 
generalized complex structures of any type, respectively.
Moreover, nilmanifolds endowed with SnN complex structures admitting neutral Calabi-Yau metrics with interesting deformation properties are given in \cite{LU-pK}.

A first step to investigate the properties of nilpotent Lie algebras with SnN complex structure is carried out in \cite{LUV-SnN}, where several algebraic constraints to their existence in terms of the ascending central series of~$\frg$ are found.
For instance,
the nilpotency step $s$ of the Lie algebra $\frg$ satisfies $s\geq 3$,
and
the dimension of the center is bounded by $\dim Z(\frg)\leq n-3$, where $2n=\dim \frg\geq 8$.
Somehow, the existence of SnN complex structures $J$ on $\frg$ 
seems to require a large nilpotency step and a small center, which gives the idea that these $J$'s might be very twisted.

It is known that SnN complex structures do not exist in real dimension $\leq 4$, and that there are only two non-isomorphic 6-dimensional nilpotent Lie algebras admitting such structures
(see \cite{COUV} for details).
However, there is no classification in higher (even) dimensions, and  
our goal in this paper is to provide the  complete list of 8-dimensional nilpotent Lie algebras $\frg$ that admit SnN complex structures up to isomorphism.
We recall that real nilpotent Lie algebras are classified only up to dimension 7. So, our method will consist in two steps: in the first one we will obtain a classification of SnN complex structures up to equivalence; then, we will achieve the classification of real Lie algebras admitting such structures.
Notice that, by the previous upper bound for $Z(\frg)$, 
these Lie algebras must have $1$-dimensional center.

\vskip.2cm

Let $\frg$ be an 8-dimensional nilpotent Lie algebra, NLA for short, with $1$-dimensional center admitting a complex structure $J$. 
We first prove that there is a partition into two families (labeled as~I and~II) according to the value of an algebraic invariant associated to the pair $(\frg,J)$ 
(see Proposition~\ref{dolb-0-1} and Definition~\ref{familias-I-II}). This allows us to find 
an appropriate reduction of the 
structure equations that is suitable for their classification up to equivalence. Indeed, 
Theorem~\ref{main-theorem}
provides a classification of complex structures~$J$ on 8-dimensional NLAs~$\frg$ with $1$-dimensional center.
The proof of this result is given in Section~\ref{reduc-fam-FI} for the case that~$J$ belongs to the Family~I, and in Section~\ref{reduc-fam-FII} for the Family~II.
We recall that similar results were obtained on 6-dimensional nilpotent Lie algebras
 in the case of abelian complex structures \cite{ABD} and of any other type of complex structure \cite{COUV}. Also complex parallelizable structures up to complex dimension seven are classified in~\cite{G} (see also \cite{Nakamura}). 
 
We then provide the relation between the equivalence classes of complex structures on $8$-dimensional nilpotent Lie algebras $\frg$ with $1$-dimensional center and the ascending central series of $\frg$ (see Tables~\ref{tab:valores-complejos} and~\ref{tab:valores-complejos2} in Theorem~\ref{main-theorem}). 
This is a crucial step for the proof of the following classification result, 
which is the main result of this paper:

\begin{theorem}\label{intro-main-theorem-2}
{\textbf{(Classification of nilpotent Lie algebras)}}
Let $\frg$ be an 8-dimensional NLA with $1$-dimensional center.
Then, $\frg$ has a complex structure if and only if it is isomorphic to one (and only one) in the following list:

\vskip.2cm

$
\begin{array}{rl}
&\mathfrak \frg_{1}^{\gamma} = (0^5,\, 13+15+24,\, 14-23+25,\, 16+27+\gamma\!\cdot\! 34), \
	\text{where }\gamma\in\{0,1\},\\[5pt]
&\mathfrak \frg_2^{\alpha} = (0^4,\, 12,\, 13+15+24,\, 14-23+25,\, 16+27+\alpha\!\cdot\! 34), \
	\text{where } \alpha\in\mathbb R, \\[5pt]
&\mathfrak \frg_3^{\gamma} = (0^4,\,12,\,
	13+\gamma\!\cdot\! 15+25,\, 15+24+\gamma\!\cdot\! 25,\, 16+27), \
	\text{ where } \gamma\in\{0,1\}, \\[6pt]
&\mathfrak \frg_4^{\alpha,\,\beta} = (0^4,\,12,\, 15+(\alpha\!+\!1)\!\cdot\! 24,\,
	\, (\alpha\!-\!1)\!\cdot\! 14-23+(\beta\!-\!1)\!\cdot\! 25,\,
	16+27+34-2\!\cdot\! 45), \\[4pt]
	& \hskip1cm \text{ where } (\alpha, \beta)\in\mathbb R^*\times \mathbb R^{+} \text{ or }\
	\mathbb R^{+}\times \{0\}, \\[5pt]
&\mathfrak \frg_5 = (0^4,\,2\!\cdot\! 12,\,14-23,\,13+24,\,16+27+35),\\[5pt]
&\mathfrak \frg_6 = (0^4,\,2\!\cdot\! 12,\,14+15-23,\,13+24+25,\,16+27+35),\\[5pt]
&\mathfrak \frg_7 = (0^5,\,15,\,25,\,16+27+34),\\[5pt]
&\mathfrak \frg_8 = (0^4,\,12,\,15,\,25,\,16+27+34),\\[5pt]
&\mathfrak \frg_9^{\gamma} = (0^3,\, 13,\, 23,\, 35,\, \gamma\!\cdot\! 12-34,\, 16+27+45),\
	\text{where }\gamma\in\{0,1\},\\[5pt]
&\mathfrak \frg_{10}^{\gamma} = (0^3,\, 13,\, 23,\, 14+25,\, 15+24,\, 16+ \gamma\!\cdot\! 25 +27),\
	\text{where }\gamma\in\{0,1\},\\[5pt]
&\mathfrak \frg_{11}^{\alpha, \beta} = \big(0^3,\, 13,\, 23,\, 14+25-35,\, \alpha \!\cdot\! 12+15+24+34,\, 16+27-45-\beta(2\!\cdot\! 25 + 35)\big),\\[4pt]
& \hskip1cm \text{ where } (\alpha, \beta)=(0,0), (1,0),(0,1) \text{ or }\ (\alpha, 1) \text{ with } \alpha\in\mathbb R^{+}, \\[5pt]
&\mathfrak \frg_{12}^{\gamma} = (0^2,\, 12,\, 13,\, 23,\, 14+25,\, 15+24,\, 16+27+ \gamma\!\cdot\! 25),\
	\text{where }\gamma\in\{0,1\}.
\end{array}
$
\end{theorem}

For the description of the nilpotent Lie algebras in Theorem~\ref{intro-main-theorem-2} we use the standard abbreviated notation (see Notation~\ref{notacion-lista-NLAs} for details).
We note that the first eight families of Lie algebras, i.e. $\frg_1^{\gamma},\ldots,\frg_8$, are those having complex structures in Family I, whereas the complex structures on $\frg_9^{\gamma},\ldots,\frg_{12}^{\gamma}$ belong to Family~II.
The list above
is ordered according to the ascending type of the nilpotent Lie algebras.
The proof of
Theorem~\ref{intro-main-theorem-2} is given in Section~\ref{clasif-real-FamI-II} (see Section~\ref{clasif-real-FamI} for the Lie algebras underlying Family I and Section~\ref{clasif-real-FamII} for those underlying Family II).

\vskip.2cm

The precise relation between the Lie algebras in Theorem~\ref{intro-main-theorem-2} and the classification of complex structures (Theorem~\ref{main-theorem}) can be found in Tables~\ref{tab:cambios-FI} and~\ref{tab:cambios-FII}. Hence, given a real nilpotent Lie algebra $\frg$ in Theorem~\ref{intro-main-theorem-2}, the information provided in these tables allows to construct the whole space of complex structures $J$ on $\frg$ up to equivalence.

\vskip.2cm

An important (and still open) problem in the geometry of complex nilmanifolds $(M,J)$ is whether their Dolbeault cohomology groups are or not canonically isomorphic to the Lie-algebra Dolbeault cohomology of the underlying pair $(\frg,J)$. In \cite{CF}, \cite{CFGU-dolbeault}, \cite{FRR}, \cite{R2} and \cite{RTW}, several steps towards a positive answer to this question are given. The results in these papers require that $(M,J)$ satisfies some special properties, which in turn force the pair $(\frg,J)$ to satisfy some algebraic constraints. Here we focus on the result obtained in \cite{FRR}, where it is required that the complex nilmanifold $(M,J)$ is suitably foliated in toroidal groups. In this setting one needs the existence of a non-trivial abelian $J$-invariant ideal $\frf$ in the nilpotent Lie algebra~$\frg$. 
It was first proved in \cite{LU-procBulgaria} that in general such an ideal $\frf$ may not exist. 
In Section~\ref{ideales-abelianos} we study this existence problem on nilpotent Lie algebras up to eight dimensions, obtaining a classification of those NLAs $\frg$ that admit a complex structure $J$ having a non-trivial abelian $J$-invariant ideal.

\section{Complex structures on nilpotent Lie algebras}

\noindent In this section we recall some results about real nilpotent Lie algebras (NLA for short)
endowed with complex structures, paying special attention to real dimension~$8$.

\medskip
Let $\frg$ be a real Lie algebra of dimension $2n$.
Its \emph{ascending central series} $\{\frg_k\}_{k\geq 0}$ is given by
$\frg_0=\{0\}$, and
$$
\frg_k=\{X\in\frg \mid [X,\frg]\subseteq \frg_{k-1}\},
$$
for any $k\geq 1$.
In particular, $\frg_1=Z(\frg)$ is the center of~$\frg$.
A Lie algebra $\frg$ is said to be \emph{nilpotent} if there is
an integer~$s\geq 1$ such that~$\frg_k=\frg$, for every~$k\geq s$.
In such case, the smallest integer $s$
satisfying the previous condition is called \emph{nilpotency step} of~$\frg$, and the Lie algebra is said to be \emph{$s$-step nilpotent}.
Thus, any nilpotent Lie algebra $\frg$ has an associated $s$-tuple
$$(m_{1},\ldots,m_{s-1},m_s) :=\left(\text{dim\,}\frg_{1},\ldots, \text{dim\,}\frg_{s-1},\text{dim\,}\frg_{s}\right)$$
which strictly increases as $0<m_{1}<\cdots <m_{s-1}<m_{s}=2n$.
We will say that $(m_{1},\ldots,m_{s})$ is the \emph{ascending type} of~$\frg$.

Obviously, NLAs with different ascending types are non-isomorphic. However, the converse is only true up to real dimension 4.
Indeed, there are three non-isomorphic 4-dimensional NLAs whose ascending types are $(4)$, $(2,4)$, and $(1,2,4)$.
In contrast, there exist four non-isomorphic 6-dimensional NLAs with the same ascending type $(2,6)$,  for instance.

\smallskip
Let $J$ be a complex structure on an NLA $\frg$, that is,
an endomorphism $J\colon\frg\longrightarrow\frg$
fulfilling $J^2=-\textrm{Id}$ and the integrability condition
$$
N_J(X,Y):=[X,Y]+J[JX,Y]+J[X,JY]-[JX,JY]=0,
$$
for all $ X,Y\in\frg$.
Observe that the terms~$\frg_{k}$ in the ascending central series
may not be invariant under~$J$. For this reason, a new series~$\{\fra_{k}(J)\}_{k}$ adapted to the complex structure~$J$
is introduced in~\cite{CFGU-dolbeault}:
$$
\left\{\begin{array}{l}
\fra_0(J)=\{0\}, \text{ and } \\[4pt]
\fra_k(J)=\{X\in\frg \mid [X,\frg]\subseteq \fra_{k-1}(J)\ {\rm and\ } [JX,\frg]\subseteq \fra_{k-1}(J)\}, \text{ for } k\geq 1.
\end{array}\right.
$$
This series $\{\fra_{k}(J)\}_{k}$
is called the \emph{ascending $J$-compatible series of~$\frg$}.
Observe that every $\fra_k(J)\subseteq\frg_{k}$ is a $J$-invariant ideal of $\frg$, and
$\fra_1(J)$ is indeed the largest subspace of the center $\frg_1$ which is $J$-invariant.

Depending on the
behaviour of the series~$\{\fra_{k}(J)\}_{k}$,
complex structures on NLAs can be classified into different types:


\begin{definition}\label{tipos_J}
\cite{CFGU-dolbeault, LUV-SnN}
{\rm
A complex structure $J$ on a nilpotent Lie algebra $\frg$ is said to be
\begin{itemize}
\item[(i)] \emph{strongly non-nilpotent}, or \emph{SnN} for short, if $\fra_1(J)=\{0\}$;

\smallskip
\item[(ii)] \emph{quasi-nilpotent}, if $\fra_1(J)\neq\{0\}$; moreover, $J$ is called
 \begin{itemize}
 \item[(ii.1)] \emph{nilpotent}, if there exists an integer~$t>0$ such that~$\fra_t(J)=\frg$,
 \item[(ii.2)] \emph{weakly non-nilpotent}, if there is an integer~$t>0$ satisfying~$\fra_t(J)=\fra_l(J)$, for every~$l\geq t$,
           and~$\fra_t(J)\neq\frg$.
 \end{itemize}
\end{itemize}
}
\end{definition}

\begin{remark}\label{otros-autores}
{\rm
Some algebraic constraints to the existence of complex structures on an NLA $\frg$ are studied in \cite{Mi} in terms of the \emph{descending} central series of $\frg$.  These constraints imply some estimates on the nilpotency step $s$; in particular, $\frg$ cannot be filiform \cite{GR}. We also recall that quasi-filiform Lie algebras of dimension $\geq 8$ do not admit any complex structure \cite{GVR}.
Hence, $s\leq 2n-3$ for any NLA~$\frg$ of dimension $2n\geq 8$ endowed with a complex structure.
}
\end{remark}

\begin{remark}
{\rm
In the recent paper~\cite{GZZ}, the structure of Lie algebras endowed with a
\emph{maximal nilpotent}~$J$ is studied.
Such $J$'s are defined as those nilpotent complex structures for which
$t=n$, where $t$ is the smallest integer such that $\fra_t(J) = \frg$ in the ascending $J$-compatible series of~$\frg$.
}
\end{remark}

One can see that quasi-nilpotent complex structures on
NLAs of a given dimension can be constructed from
other complex structures defined on (strictly) lower dimensional NLAs (see \cite[Section~2]{LUV-SnN} for details).
Therefore, the essentially new complex structures that arise in each even real dimension are those of
strongly non-nilpotent type. That is to say, SnN complex structures constitute the fundamental piece to fully understand the class of NLAs endowed with 
complex structures.

In real dimension~$4$ it is well known that SnN complex structures do not exist, whereas
in dimension~$6$ one has the following result:

\begin{theorem}\label{structural-dim6}\cite{COUV}
Let $\frg$ be an NLA of real dimension $6$. If $\frg$ admits an SnN complex structure, then its ascending type is
$(\dim\frg_k)_k=$~$(1,3,6)$ or $(1,3,4,6)$.
\end{theorem}

More generally, all the pairs $(\frg,J)$ with $\text{dim}\,\frg=6$ 
are classified in \cite{COUV} by means of their complex structure equations. There are only two NLAs, up to isomorphism, admitting SnN complex structures.

Concerning higher dimensions, \cite{LUV-SnN} provides
several general restrictions on the terms of the ascending central series of NLAs admitting
SnN complex structures.
Among them, we highlight the following one:

\begin{theorem}\label{prop_centro}\cite[Theorem 3.11]{LUV-SnN}
Let $(\frg, J)$ be a $2n$-dimensional nilpotent Lie algebra, with $n\geq 4$, endowed with a strongly
non-nilpotent complex structure $J$. Then, $1\leq \dim\frg_1 \leq n-3$.
\end{theorem}

From Definition~\ref{tipos_J} one can clearly deduce that any complex structure on an
NLA with $1$-dimensional center is of SnN type.
Thanks to Theorem~\ref{prop_centro}, the converse also holds in eight dimensions:

\begin{corollary}\label{cor-1-dim-center}
Let $\frg$ be an $8$-dimensional NLA admitting a
complex structure $J$. The following properties are equivalent:
\begin{itemize}
\item the center of $\frg$ has dimension $1$;
\item the complex structure $J$ is strongly non-nilpotent.
\end{itemize}
\end{corollary}

A structural result in the spirit of Theorem~\ref{structural-dim6} is available in eight dimensions:

\begin{theorem}\label{teorema-estructura-acs}\cite[Theorem 4.1]{LUV-SnN}
Let $\frg$ be an NLA of real dimension $8$. If $\frg$ admits an SnN complex structure, then its ascending type is
$(\dim\frg_k)_k=$~$(1,3,8)$, $(1,3,5,8)$, $(1,3,6,8)$, $(1,3,5,6,8)$, $(1,4,8)$, $(1,4,6,8)$, $(1,5,8)$, or $(1,5,6,8)$.
\end{theorem}

Furthermore, the complex structure equations for any pair $(\frg,J)$ with~$\dim \frg=8$ and~$\fra_1(J)=\{0\}$ are given in \cite{LUV-SnN}. 
Before presenting them, we need to recall some basic concepts.

Let $\gc^*$ denote the dual of the complexification $\gc$ of $\frg$.
Then, there is a natural bigraduation induced on $\bigwedge^* \,\gc^* =\oplus_{p,q}
\bigwedge^{p,q}_{J}(\frg^*)$, where the spaces $\bigwedge^{1,0}_{J}(\frg^*)$
and $\bigwedge^{0,1}_{J}(\frg^*)$ are, respectively, the eigenspaces of the
eigenvalues $\pm i$ of $J$ as an endomorphism of~$\gc^*$. For simplicity,
we will denote $\bigwedge^{p,q}_{J}(\frg^*)$ by $\frg^{p,q}_J$.
Let $d\colon \bigwedge^* \gc^* \longrightarrow \bigwedge^{*+1}
\gc^*$ be the extension to the complexified exterior algebra of the
usual Chevalley-Eilenberg differential. 
Since $J$ is a complex structure, 
we have that $\pi_{0,2} \circ d\vert_{\frg^{1,0}_J}
\equiv 0$, where $\pi_{0,2}\colon \bigwedge^{2} \gc^*
\longrightarrow \frg^{0,2}_J$ denotes the canonical projection.
In fact, this is equivalent to the integrability condition $N_J\equiv 0$.
Under these assumptions the differential $d$ splits as 
$d=\partial+\db$, where $\db\colon \frg^{p,q}_J \longrightarrow \frg^{p,q+1}_J$ is defined by $\db=\pi_{p,q+1} \circ d\vert_{\frg^{p,q}_J}$, and $\partial$ is the conjugate of $\db$.
From $d^2=0$ we have $\db^2=0$, and the associated \emph{Lie-algebra Dolbeault cohomology} is given by  
\begin{equation}\label{Dolbeault}
H^{p,q}_{\db}(\frg,J)=
   \mathrm{Ker}\{\db\colon \frg^{p,q}_J \longrightarrow \frg^{p,q+1}_J\} /\, 
   \mathrm{Im}\{\db\colon \frg^{p,q-1}_J \longrightarrow \frg^{p,q}_J\}.
\end{equation}

Let $\{\omega^k\}_{k=1}^n$ be any basis of~$\frg^{1,0}_J$. 
From the integrability condition of $J$ we have 
\begin{equation}\label{complex-ecuss}
d \omega^k = \sum_{1\leq r<s\leq n} A^k_{r s}\,\omega^{r s} + \sum_{1\leq r,s\leq n} B^k_{r\bar{s}}\,\omega^{r\bar{s}},\quad 1\leq k\leq n,
\end{equation}
for certain $A^k_{r\,s},B^k_{r\,\bar{s}}\in \C$. Here, and in the rest of the paper, we denote by $\omega^{jk}$, resp. $\omega^{j\overline{k}}$, the wedge product
$\omega^j\wedge\omega^k$, resp.
$\omega^j\wedge\omega^{\overline{k}}$, where
$\omega^{\overline{k}}$ indicates the complex conjugate of
$\omega^k$.
Since $\frg$ is an NLA, by~\cite{S} one can take the basis
$\{\omega^k\}_{k=1}^n$ so that  
\begin{equation}\label{Salamon-ecuss}
d\omega^1=0  \quad\hbox{ and }\quad  d \omega^{k} \in \mathcal{I}(\omega^1,\ldots,\omega^{k-1}), \quad
\mbox{ for } 2\leq k\leq n,
\end{equation}
where $\mathcal{I}(\omega^1,\ldots,\omega^{k-1})$ is the ideal in
$\bigwedge\phantom{\!}^* \,\gc^*$ generated by
$\{\omega^1,\ldots,\omega^{k-1}\}$.


Note that one can construct $J$ by defining an appropriate space $\frg^{1,0}_J$. Even more, one can construct a pair $(\frg,J)$ by defining appropriate equations. 
More precisely, consider equations of the form \eqref{complex-ecuss} satisfying the condition \eqref{Salamon-ecuss}, and declare 
$\{\omega^k\}_{k=1}^n$ to be a basis of bidegree $(1,0)$. Then, they define a Lie algebra $\frg$ with a complex structure $J$ as long as $d^2=0$, namely, the corresponding Lie bracket satisfies the Jacobi identity. Notice that this imposes several conditions on the coefficients $A^k_{rs},B^k_{r\bar{s}}\in \C$ in \eqref{complex-ecuss}. Since we are interested in defining SnN complex structures on $8$-dimensional NLAs, we need to fix $n=4$ and pay attention to the dimension of the center of $\frg$. This motivates the following definition:

\begin{definition}\label{admissible}
{\rm 
The coefficients $A^k_{rs},B^k_{r\bar{s}}\in \C$ are said to be}
admissible 
{\rm if the equations \eqref{complex-ecuss} satisfy $d^2 = 0$ and the associated Lie algebra has $1$-dimensional center.} 
\end{definition}

\begin{proposition}\label{prop-paper1}
Let $J$ be an SnN complex structure on an 8-dimensional NLA $\frg$. 
Then, there exists a basis of $(1,0)$-forms $\{\omega^k\}_{k=1}^4$ in terms of which the complex structure equations of $(\frg,J)$ are of the form
\begin{equation}\label{ecus-generales}
\left\{
\begin{split}
d\omega^1 &= 0,\\[-4pt]
d\omega^2 &= A\,\omega^{1\bar 1} -B(\omega^{14}- \omega^{1\bar 4}),\\[-4pt]
d\omega^3 &=
  F\,\omega^{1\bar 1}
+K\,\omega^{2\bar 2}
+ C\,\omega^{12}
+D\,\omega^{1\bar 2}
 +G\,\omega^{2\bar1}
 -E\,(\omega^{14}- \omega^{1\bar 4})
 -H\,(\omega^{24}-\omega^{2\bar 4}),\\[-4pt]
d\omega^4 &= L\,\omega^{1\bar 1}+i\,s\,\omega^{2\bar 2}+i\,t\,\omega^{3\bar3}
+(M\,\omega^{1\bar 2}-\bar{M}\,\omega^{2\bar 1})
  +(N\,\omega^{1\bar 3}-\bar{N}\,\omega^{3\bar 1})+(P\,\omega^{2\bar 3}-\bar{P}\,\omega^{3\bar2}),\end{split}
 \right.
\end{equation}
where the coefficients 
$A, \ldots, P\in \C$ 
and $s,t \in\mathbb R$ are admissible. 
\end{proposition}

\begin{proof}
In \cite{LUV-SnN} admissible complex equations are obtained 
depending on the dimension of the second term~$\frg_2$ in the ascending central series of any 8-dimensional NLA $\frg$ endowed with an SnN complex structure $J$ (see Propositions 4.12, 4.13 and 4.14 in \cite{LUV-SnN} for the three possible cases $\mathrm{dim}\,\frg_2=3,4$ or~$5$, respectively). 
We here simply note that we can gather those equations 
in the more general setting provided by~\eqref{ecus-generales}. 
\end{proof}

\section{Classification of SnN complex structures in dimension $8$}\label{complex-classif}

\noindent In this section, we classify the SnN complex structures on $8$-dimensional NLAs up to equivalence. 
Let $\frg$ and $\frg'$ be two Lie algebras endowed with respective complex structures $J$ and $J'$.
They are said to be {\it equivalent} if there is an isomorphism of Lie algebras
$f\colon \frg\longrightarrow\frg'$ such that $J=f^{-1}\circ J'\circ f$.
That is, if there exists a
$\mathbb{C}$-linear isomorphism $F(:=\!\!f^*)\colon \frg'^{1,0}_{J'}
\longrightarrow \frg^{1,0}_{J}$ such that $d_{\frg} \circ F=F \circ d_{\frg'}$,
where $d_{\frg}$ and $d_{\frg'}$ are the (extended) Chevalley-Eilenberg differentials of $\frg$ and $\frg'$, respectively. 
We will usually denote both differentials by the same letter $d$.

Note that the equivalence above induces an isomorphism 
$F\colon H^{p,q}_{\db}(\frg',J')  \longrightarrow H^{p,q}_{\db}(\frg,J)$
for every $p,q$. In the following result we study the invariant given by the Lie-algebra Dolbeault cohomology group of bidegree~$(p,q)=(0,1)$.

\begin{proposition}\label{dolb-0-1}
For any SnN complex structure $J$ on an 8-dimensional NLA $\frg$, 
the dimension of the Lie-algebra Dolbeault cohomology group $H^{0,1}_{\db}(\frg,J)$ is either $2$ or $3$. 
\end{proposition}

\begin{proof}
From \eqref{Dolbeault} we have  
$H^{0,1}_{\db}(\frg,J)=\mathrm{Ker}\{\db\colon \frg^{0,1}_J \longrightarrow \frg^{0,2}_J\}$. By Proposition~\ref{prop-paper1} we can take a basis $\{\omega^k\}_{k=1}^4$ of $(1,0)$-forms satisfying \eqref{ecus-generales} for some tuple $A, \ldots, P\in \C$ 
and $s,t \in\mathbb R$ of admissible coefficients. 
Clearly, $\db\omega^{\bar 1}=0=\db\omega^{\bar 4}$, 
so $\dim H^{0,1}_{\db}(\frg,J) \geq 2$. 

If $\dim H^{0,1}_{\db}(\frg,J)= 4$, then $\db\omega^{\bar 2}=0=\db\omega^{\bar 3}$ and this implies $B=C=E=H=0$ in the equations~\eqref{ecus-generales}. 
However, in this case $U=\Real(Z_4)$ and 
$JU=-\Imag(Z_4)$ 
would belong to the center of $\frg$, being $Z_4$ the dual of $\omega^4$. This
is a contradiction to the fact that the tuple of coefficient is admissible, so one concludes that $2\leq \dim H^{0,1}_{\db}(\frg,J) \leq 3$. 
\end{proof}

This result provides a partition of the space of SnN complex structures $J$ into two 
families:

\begin{definition}\label{familias-I-II}
{\rm
We say that $J$ belongs to \emph{Family I} (resp. \emph{Family II}) if the invariant $H^{0,1}_{\db}(\frg,J)$ has maximal dimension, i.e. equal to $3$ (resp. minimal dimension, i.e. equal to $2$).
}
\end{definition}

The main goal of this section is to prove the classification result below.
Recall that by Corollary~\ref{cor-1-dim-center},
a complex structure on an 8-dimensional NLA is SnN if and only if the center of the NLA is 1-dimensional.

\begin{theorem}\label{main-theorem}
{\textbf{(Classification of complex structures)}}
Let $J$ be a complex structure on an 8-dimensional NLA $\frg$ with $1$-dimensional center.
Then, there exists a basis of $(1,0)$-forms $\{\omega^k\}_{k=1}^4$ in terms of which the complex structure equations of $(\frg,J)$ 
are one (and only one) of the following:
\begin{itemize}
\item[(i)] if $J$ belongs to {\textbf{Family I}}, then
\begin{equation}\label{FI-SnN}
\left\{
\begin{split}
d\omega^1 &= 0,\\[-4pt]
d\omega^2 &= \varepsilon\,\omega^{1\bar 1},\\[-4pt]
d\omega^3 &= \omega^{14}+\omega^{1\bar 4}+a\,\omega^{2\bar 1}+ i\,\delta\,\varepsilon\,b\,\omega^{1\bar 2},\\[-4pt]
d\omega^4 &= i\,\nu\,\omega^{1\bar 1} +b\,\omega^{2\bar 2}+ i\,\delta\,(\omega^{1\bar 3}-\omega^{3\bar 1}),
\end{split}
\right.
\end{equation}
where $\delta=\pm 1$, $(a,b)\in \mathbb R^2-\{(0,0)\}$ with $a\geq 0$,
and the tuple 
$(\varepsilon, \nu, a, b)$ takes the following values:

\vskip.1cm

$(0,0,0,1),\, (0,0, 1, 0),\, (0,0, 1, 1),\, (0,1, 0, \nicefrac{b}{|b|}),\, (0,1, 1,b),\,
(1,0, 0,1),\,  (1,0, 1,|b|) \text{ or }(1,1, a, b)$.

\vskip.1cm

\noindent Moreover, the ascending type of $\frg$ is $(\mathrm{dim}\,\frg_k)_{k}=(1,3,8)$, $(1,3,6,8)$, $(1,4,8)$,
$(1,4,6,8)$, $(1,5,8)$, or $(1,5,6,8)$,
and the relation between the 
parameters in \eqref{FI-SnN} and the ascending type of $\frg$ is given in Table~\ref{tab:valores-complejos}.
\begin{table}[h]
\centering
\renewcommand{\arraystretch}{1.2}
\renewcommand{\tabcolsep}{0.3cm}
\begin{tabular}{|c||c|c|c|c|c|}
\hline
$(\mathrm{dim}\,\frg_k)_{k}$ & $\varepsilon$ & $\nu$ & $a$ & $b$ &$\delta$\\
\hline\hline
\multirow{2}{*}{$(1,3,8)$} & \multirow{2}{*}{$0$} & \multirow{2}{*}{$0$} & \multirow{2}{*}{$1$} & $0$&\multirow{2}{*}{$\pm 1$} \\
\cline{5-5}
 & & & & $1$& \\
\hline\hline
\multirow{3}{*}{$(1,3,6,8)$} & $1$ &$0$& \multirow{2}{*}{$1$} & $b\geq 0$ &\multirow{3}{*}{$\pm 1$} \\
\cline{2-3}\cline{5-5}
 &$0$ & $1$ & & \multirow{2}{*}{$b\in\mathbb R$}& \\
\cline{2-4}
 & $1$ & $1$ & $a>0$ && \\
\hline\hline
$(1,4,8)$ & $1$ & $1$ &$0$ & $2\delta$ &$\pm 1$\\
\hline\hline
\multirow{2}{*}{$(1,4,6,8)$} & \multirow{2}{*}{$1$} & $0$ & \multirow{2}{*}{$0$} & $1$ &\multirow{2}{*}{$\pm 1$}\\
\cline{3-3}\cline{5-5}
 & & $1$ & & $b\in\mathbb R -\{0,2\delta\}$ &\\
\hline\hline
$(1,5,8)$ & $0$ & $0$ & $0$ &$1$ &$\pm 1$ \\
\hline\hline
\multirow{2}{*}{$(1,5,6,8)$} & \multirow{2}{*}{$0$} & \multirow{2}{*}{$1$} & \multirow{2}{*}{$0$} & $-1\ \ $  &\multirow{2}{*}{$\pm 1$}\\
\cline{5-5}
 & & & & $1$& \\
\hline
\end{tabular}
\medskip
\caption{Complex structures in Family I up to equivalence}
\label{tab:valores-complejos}
\end{table}


\item[(ii)] if $J$ belongs to {\textbf{Family II}}, then
\begin{equation}\label{FII-SnN}
\left\{
\begin{split}
d\omega^1&=0,\\[-4pt]
d\omega^2&=\omega^{14}+\omega^{1\bar 4},\\[-4pt]
d\omega^3&=a\,\omega^{1\bar 1}
                 +\varepsilon\,(\omega^{12}+\omega^{1\bar 2}-\omega^{2\bar 1})
                 +i\,\mu\,(\omega^{24}+\omega^{2\bar 4}),\\[-4pt]
d\omega^4&=i\,\nu\,\omega^{1\bar 1}-\mu\,\omega^{2\bar 2}+i\,b\,(\omega^{1\bar 2}-\omega^{2\bar 1})+i\,(\omega^{1\bar 3}-\omega^{3\bar 1}),
\end{split}
\right.
\end{equation}
where $a, b\in\mathbb R$, and the tuple $(\varepsilon, \mu, \nu, a, b)$ takes the following values:

\vskip.1cm

$(1, 1, 0, a, b),\, (1, 0, 1, a, b),\, (1, 0, 0, 0, b), (1, 0, 0, 1, b),\, (0, 1, 0, 0, 0) \text{ or } (0, 1, 0, 1, 0)$.

\vskip.1cm

\noindent Moreover, the ascending type of $\frg$ is $(\mathrm{dim}\,\frg_k)_{k}=(1,3,5,8)$ or $(1,3,5,6,8)$,
and the relation between the 
parameters in \eqref{FII-SnN} and the ascending type of $\frg$ is given in  Table~\ref{tab:valores-complejos2}.

\begin{table}[h]
\centering
\renewcommand{\arraystretch}{1.3}
\renewcommand{\tabcolsep}{0.3cm}
\begin{tabular}{|c||c|c|c|c|c|}
\hline
$(\mathrm{dim}\,\frg_k)_{k}$ & $\varepsilon$ & $\mu$ & $\nu$ &$a$ & $b$ \\
 \hline\hline
\multirow{3}{*}{$(1,3,5,8)$} & $0$ &$1$ & $0$ &\multirow{2}{*}{$0,\, 1$} & $0$ \\
\cline{2-4}\cline{6-6}
 &$1$ & $0$ & $0$&& \multirow{2}{*}{$b\in\mathbb R$} \\
\cline{2-5}
 & $1$ & $1$ & $0$ & $a\in\mathbb R$ &\\
\hline\hline
$(1,3,5,6,8)$ & $1$ & $0$ &$1$&$a\in\mathbb R$ & $b\in\mathbb R$ \\
\hline
\end{tabular}
\medskip
\caption{Complex structures in Family II up to equivalence}
\label{tab:valores-complejos2}
\end{table}
\end{itemize}
\end{theorem}

The rest of this section is devoted to the proof of Theorem~\ref{main-theorem}.
Starting from the complex structure equations 
given in Proposition~\ref{prop-paper1},  we will arrive at an appropriate reduction (see Proposition~\ref{prop-resumen} below) that is suitable for the classification of SnN complex structures. 


Firstly, we obtain several conditions derived from the fact that the coefficients 
$A, \ldots, P\in \C$ 
and $s,t \in\mathbb R$ in~\eqref{ecus-generales} are admissible (see Definition~\ref{admissible}). 
We notice that the Jacobi identity,
i.e. $d^2 \omega^k=0$ for $1 \leq k \leq 4$,
is equivalent to the following equations:
\begin{equation}\label{jacobi-general1}
\begin{array}{lll}
AH-BG+\bar BD=0,&\quad  AK = BK = 0,&\quad  t\, H = t\, K=t\, C = 0,\\[5pt]
K\bar N - P\bar C- \bar PG=0,&\quad   H\, \Real L = 0,&\quad   t\, D = t\, G=0,\\[5pt]
isA - F\bar P - N\bar C+ \bar N D=0,&\quad    \Real (P\bar H)=0,&\quad   itE+BP=0,\\[5pt]
isB-E\bar P - N\bar H=0,&\quad   \Real(M\bar B + N\bar E)=0,&\quad   itF + AP=0.
\end{array}
\end{equation}
For the condition on the center,
let us denote by $\{Z_k \}_{k=1}^4$ the dual basis to $\{\omega^k\}_{k=1}^4$.
Using the well-known formula $d \alpha(X,Y)=-\alpha([X,Y])$, for any $\alpha \in \frg^*$ and $X,Y \in \frg$,
and its extension to the complexification, it is easy to check from \eqref{ecus-generales}
that
\begin{equation}\label{ReZ4}
[X,Z_4+\bar Z_4]=0
\end{equation}
for any $X \in \frg$. Since the center of $\frg$ is $1$-dimensional, necessarily $\frg_1=\langle \Real  Z_4 \rangle$.
Furthermore, it is clear from equations~\eqref{ecus-generales} that
the vanishing of the tuples $(B, E, H)$, $(N, P, t)$, or $(C, D, G, H, K, M, P, s)$,
implies $\Imag Z_4\in\frg_1$,
$\langle \Real Z_3,\, \Imag Z_3 \rangle \subset \frg_1$, or
$\langle\Real Z_2,\, \Imag Z_2 \rangle \subset \frg_1$, respectively, which would give a contradiction to $\dim\frg_1=1$.
Hence,
the following conditions must be satisfied:
\begin{equation}\label{negative-conditions}
(B, E, H)\neq (0,0,0),\quad (N, P, t)\neq (0,0,0),\quad (C, D, G, H, K, M, P, s)\neq (0,\ldots,0).
\end{equation}

\medskip

In what follows, we consider~\eqref{ecus-generales} bearing in mind the conditions \eqref{jacobi-general1} and \eqref{negative-conditions}.
As noticed above, for the classification up to equivalence, one can study 
$\mathbb{C}$-linear isomorphisms $F\colon \frg'^{\,1,0}_{J'}
\longrightarrow \frg^{1,0}_{J}$ commuting with the differentials, i.e. $d \circ F=F \circ d$.
Thus, whenever an equivalence exists, we will construct it by means of an 
explicit change of $(1,0)$-bases.

\begin{lemma}\label{lema-t-0}
In the equations \eqref{ecus-generales}, one can assume $t=0$.
\end{lemma}

\begin{proof}
Let us suppose that $t \neq 0$ in equations~\eqref{ecus-generales}.
By~\eqref{jacobi-general1}, we get $C = D = G = H = K = 0$.
Hence, conditions \eqref{jacobi-general1} and \eqref{negative-conditions} reduce to
\begin{equation}\label{jacobi-general-11}
\begin{array}{llll}
isA - F\bar P =0,&\quad itE+BP=0, &\quad \Real(M\bar B + N\bar E)=0,&\quad (N, P, t)\neq (0,0,0),
\\[5pt]
isB-E\bar P=0,&\quad  itF + AP=0,  &\quad  (B, E)\neq (0,0),&\quad (M, P, s)\neq (0,0,0).
\end{array}
\end{equation}

We first observe that $B\neq 0$, as otherwise the condition $itE=0$ would imply $E=0$,
which gives a contradiction to $(B, E)\neq (0,0)$. Bearing this in mind, we now consider two cases.

On the one hand, if $E=0$ then~\eqref{jacobi-general-11} implies $F=P=s=0$.
Taking $\tau^k = \omega^k$, for $k=1,4$, $\tau^2 = \omega^3$ and $\tau^3 = \omega^2$,
one directly gets equations of the form~\eqref{ecus-generales} for the new (1,0)-basis $\{\tau^k\}_{k=1}^4$
with $t_{\tau}=0$. We are denoting by $t_{\tau}$ the coefficient of $\tau^{3\bar3}$ in the equation $d \tau^4$.   

On the other hand, for $E\neq 0$  we consider the $(1,0)$-basis $\{\tau^k\}_{k=1}^4$ defined by
$\tau^k = \omega^k$, for $k=1, 3, 4$, and $\tau^2 = E\,\omega^2 - B\,\omega^3$. 
Then, the structure equations in terms of $\{\tau^k\}_{k=1}^4$ are again of the
form~\eqref{ecus-generales}.
Using~\eqref{jacobi-general-11} it can be directly seen that
the coefficient $t_{\tau}$ satisfies
$$
t_{\tau} = i\,s \frac{|B|^2}{|E|^2} + i\,t +\frac{B\, P}{E} -\frac{\bar{B}\, \bar{P}}{\bar{E}}
= \frac{\bar{B}}{|E|^2} (isB-E\bar P) + \frac{1}{E} (itE+B P) =0.
$$
\end{proof}

\begin{lemma}\label{lema-P-igual-0}
In the equations \eqref{ecus-generales}, in addition to $t=0$, we can also set $P=K=0$.
\end{lemma}

\begin{proof}
By Lemma~\ref{lema-t-0} we can assume $t=0$.
If we suppose that $P\neq 0$, then by \eqref{jacobi-general1} 
we immediately get $A=B=0$.
Hence,
\eqref{jacobi-general1} and \eqref{negative-conditions} are simplified
to the following conditions:
\begin{equation}\label{jacobi-general21}
\begin{array}{lll}
K\bar N - P\bar C- \bar P G=0, & \ \quad H\,\Real L = 0, & \ \quad (E, H)\neq (0,0),\\[5pt]
F\bar P + N\bar C - \bar N D=0, & \ \quad \Real (P\bar H)=0, & \ \quad (N, P)\neq (0,0),\\[5pt]
E\bar P + N\bar H=0, & \ \quad \Real(N\bar E)=0, & \ \quad (C, D, G, H, K, M, P, s)\neq (0,\ldots,0).
\end{array}
\end{equation}

Since $E=-N\bar{H}/\bar{P}$,
the condition $(E,H)\neq (0,0)$ in~\eqref{jacobi-general21}
implies $H\neq 0$. In turn, this gives $\Real L=0$, again by~\eqref{jacobi-general21}.
We consider the $(1,0)$-basis $\{\tau^k\}_{k=1}^4$ defined by
$$
\tau^1 = N\,\omega^1 + P\,\omega^2 , \qquad \tau^2 = \omega^1, \qquad \tau^k = \omega^k,\quad k= 3, 4.
$$
Using $\Real L=0$, a direct calculation shows that the structure equations in terms of $\{\tau^k\}_{k=1}^4$ are again of the
form~\eqref{ecus-generales}, with corresponding coefficients $t_{\tau}=0$ and $P_{\tau}=0$ in the equation $d \tau^4$.

Finally, now that we have $P=t=0$, it suffices to use the second equation of~\eqref{jacobi-general1} to obtain $K=0$, as $N\neq 0$
by~\eqref{negative-conditions}.
\end{proof}

Taking into account Lemmas~\ref{lema-t-0} and~\ref{lema-P-igual-0}, we have:

\begin{proposition}\label{prop-resumen}
Let $J$ be a complex structure on an 8-dimensional NLA $\frg$ with $1$-dimensional center.
Then, there exists a basis of $(1,0)$-forms $\{\omega^k\}_{k=1}^4$ in terms of which the
complex structure equations of $(\frg,J)$ are of the form:
\begin{equation}\label{ecusG}
\left\{\begin{split}
d\omega^1 &= 0,\\[-4pt]
d\omega^2 &= A\,\omega^{1\bar 1} -B(\omega^{14}- \omega^{1\bar 4}),\\[-4pt]
d\omega^3 &= F\,\omega^{1\bar 1}  + C\,\omega^{12} +D\,\omega^{1\bar 2} +G\,\omega^{2\bar1}-E\,(\omega^{14}- \omega^{1\bar 4})
  -H\,(\omega^{24}-\omega^{2\bar 4}),\\[-4pt]
d\omega^4 &= L\,\omega^{1\bar 1}+i\,s\,\omega^{2\bar 2}+(M\,\omega^{1\bar 2}-\bar{M}\,\omega^{2\bar 1})+(N\,\omega^{1\bar 3}-\bar{N}\,\omega^{3\bar 1}),
 \end{split}\right.
\end{equation}
where the coefficients $A,\ldots, N\in\C$ and $s\in\mathbb R$ are admissible; in particular, they satisfy the conditions:
\begin{equation}\label{jacobi-general}
\begin{array}{llll}
AH-BG+\bar BD=0, &\,\,  isB- N\bar H=0,&\,\,   H\,\Real L = 0, &\,\,   (B, E, H)\neq (0,0,0),\\[5pt]
 \Real(M\bar B + N\bar E)=0, &\,\,   isA - N\bar C + \bar N D=0, &\,\,   N\neq 0,  &\,\,  
	(C, D, G, H, M, s)\neq (0,\ldots,0).
\end{array}
\end{equation}
Moreover, 
the complex structure $J$ belongs to Family I (resp. Family II) if and only if  $B=0$ (resp. $B\neq 0$) in the equations \eqref{ecusG}.
\end{proposition}

\begin{proof}
Lemmas~\ref{lema-t-0} and~\ref{lema-P-igual-0} directly imply the first part of the proposition. We now prove that $\dim H^{0,1}_{\db}(\frg,J)=3$ (that is, $J$ belongs to Family I) if and only if $B=0$. By \eqref{ecusG} it is clear that 
$H^{0,1}_{\db}(\frg,J)=\langle \omega^{\bar 1},\omega^{\bar 2},\omega^{\bar 4} \rangle$ when $B$ vanishes. 
Hence, it remains to prove that $B\neq0$ implies $\dim H^{0,1}_{\db}(\frg,J)= 2$. 

Suppose $B\neq0$, and let $\lambda\omega^{\bar 2}+\mu\omega^{\bar 3}$ be $\db$-closed 
for some $\lambda,\mu\in\C$ with $(\lambda,\mu)\neq (0,0)$. From the equations \eqref{ecusG} it follows that this implies $C=H=0$. Then, by \eqref{jacobi-general} we have $D=G=s=0$, 
together with the condition $\Real(M\bar B + N\bar E)=0$, where $B,M,N\neq0$.
Let $\{Z_k \}_{k=1}^4$ be the dual basis to $\{\omega^k\}_{k=1}^4$.
A direct calculation from \eqref{ecusG} shows that both
$$
U=\Real(N \bar{Z}_2 - M  \bar{Z}_3)
\quad\text{and}\quad 
JU=-\Imag(N \bar{Z}_2 - M  \bar{Z}_3)
$$ 
belong to the center of $\frg$. However, this implies $\text{dim}\,Z(\frg)>1$, which is a contradiction. 
\end{proof}

In the following Sections~\ref{reduc-fam-FI} and~\ref{reduc-fam-FII} we study the Families I and II, respectively, 
in order to prove the parts~\textrm{(i)} and~\textrm{(ii)} of Theorem~\ref{main-theorem}.

\subsection{Study of Family I}\label{reduc-fam-FI}	

We here accomplish the study up to equivalence of those complex structures belonging to Family~I. 
We first prove that all such complex structures are parametrized by the  equations~\eqref{FI-SnN} in Theorem~\ref{main-theorem}.
Then,
we classify them up to equivalence and 
compute the ascending central series of the underlying
$8$-dimensional nilpotent
Lie algebras, reaching Table~\ref{tab:valores-complejos}.


\begin{lemma}\label{lema-previo-FI}
Let $J$ be a complex structure in Family I.
Then, there exists a basis of $(1,0)$-forms $\{\omega^k\}_{k=1}^4$ 
satisfying \eqref{ecusG} with $B=F=C=H=M=0$, and 
\begin{equation}\label{jacobi1}
\begin{array}{llll}
isA + \bar N D=0,&\quad  \Real(N\bar E)=0,&\quad NE\neq 0,&\quad (D, G, s)\neq (0,0,0).
\end{array}
\end{equation}
\end{lemma}

\begin{proof}
We first observe that if we impose $B=0$ in \eqref{jacobi-general}, then we are forced to consider $H=0$.
Consequently, equations~\eqref{jacobi-general} become
\begin{equation*}
\begin{array}{llll}
isA - N\bar C + \bar N D=0,&\quad  \Real(N\bar E)=0,&\quad NE\neq 0,&\quad (C, D, G, M, s)\neq (0,0,0,0,0).
\end{array}
\end{equation*}

Let us show that the coefficients $C, F$, and $M$
in~\eqref{ecusG} can be set equal to zero.
This can be done by defining the new $(1,0)$-basis $\{\tau^k\}_{k=1}^4$ as follows:
$$
\tau^1 = \omega^1, \quad\
\tau^2 = \omega^2, \quad\
\tau^3 = \omega^3 +\frac{\bar M}{\bar N}\,\omega^2, \quad\
\tau^4 = \omega^4 -\frac{C}{E}\,\omega^2 + \frac{M\bar A+N\bar F}{N\bar E}\,\omega^1.
$$
Indeed, in terms of $\{\tau^k\}_{k=1}^4$ the complex structure equations are of the form~\eqref{ecusG}
with new coefficients
$A_\tau, \ldots, N_\tau,s_\tau$ 
satisfying
$
B_\tau = F_\tau= C_\tau = H_\tau = M_\tau = 0.
$
Renaming the basis and the coefficients, we directly get the result.
Notice that the conditions~\eqref{jacobi-general} reduce to~\eqref{jacobi1}.
\end{proof}

\begin{lemma}\label{lema-FI-reducida1}
For any complex structure in Family I, 
there is a $(1,0)$-basis $\{\omega^k\}_{k=1}^4$ satisfying
$$
d\omega^1 = 0,\ \ 
d\omega^2 = \varepsilon\,\omega^{1\bar1},\ \  
d\omega^3 = \omega^{14}+\omega^{1\bar4}+i\,\delta\,\varepsilon\, b\,\omega^{1\bar 2} + G\,\omega^{2\bar 1},\ \ 
d\omega^4 = L\,\omega^{1\bar1} + b\,\omega^{2\bar2}+i\,\delta (\omega^{1\bar3}-\omega^{3\bar1}),
$$
with $\varepsilon\in\{0,1\}$, $\delta = \pm 1$,  $G, L\in\mathbb C$ and $b\in\mathbb R$ such that $(G,b)\neq (0,0)$.
\end{lemma}

\begin{proof}
Starting with a basis of $(1,0)$-forms $\{\omega^k\}_{k=1}^4$ as given in 
Lemma~\ref{lema-previo-FI}, we consider $\{\tau^k\}_{k=1}^4$ defined by
$\tau^k = \lambda_k\, \omega^k$, for $1\leq k\leq 4$,
where
$$
\lambda_1 = |\Imag(N\bar E)|^{\nicefrac12},\quad
\lambda_2 =
	\begin{cases}
	\begin{array}{ll}
	1,& \text{if }A=0,\\
	\frac{|\Imag(N\bar E)|}{A},& \text{if }A\neq 0,
	\end{array}
	\end{cases}
\lambda_3 = - i\,\frac{|\Imag(N\bar E)|^{\nicefrac12}}{E},\quad
\lambda_4 = i.
$$
It suffices to rename the basis and the coefficients in the corresponding structure equations in order to get the desired result.
Here, we simply note that the value $\varepsilon=0$ (resp. $\varepsilon=1$) in the statement of the lemma comes from
the case $A=0$ (resp. $A\neq 0$).
Moreover $\delta=\pm 1$, where the sign precisely
corresponds to $\sign ( \Imag(N\bar E))$.
We note that the coefficient in $\omega^{1\bar2}$ comes from the condition $d^2\omega^4=0$.
\end{proof}

Using the previous lemma, in the following result we arrive at the desired reduced structure equations~\eqref{FI-SnN}
of Theorem~\ref{main-theorem} for complex structures in the Family I.

\begin{proposition}\label{Prop-FI}
Every complex structure $J$ in Family I can be described by equations of the form
$$
d\omega^1 = 0, \ \ 
d\omega^2 = \varepsilon\,\omega^{1\bar1}, \ \ 
d\omega^3 = \omega^{14}+\omega^{1\bar4} + a\,\omega^{2\bar 1} + i\,\delta\,\varepsilon\,b\,\omega^{1\bar 2}, \ \ 
d\omega^4 = i\,\nu\,\omega^{1\bar1} + b\,\omega^{2\bar2}+i\,\delta (\omega^{1\bar3}-\omega^{3\bar1}),
$$
where $\varepsilon,\nu\in\{0,1\}$, $\delta=\pm 1$, and $a,b\in\mathbb R$ with $a\geq 0$ and $(a,b)\neq (0,0)$.
\end{proposition}

\begin{proof}
Consider the complex structure equations in Lemma~\ref{lema-FI-reducida1} in terms of a
$(1,0)$-basis $\{\sigma^k\}_{k=1}^4$ with coefficients $(\varepsilon_{\sigma},\,\delta_{\sigma},\,
b_{\sigma},\,G_{\sigma},\,L_{\sigma})$.
We first normalize the coefficient $L_{\sigma}$ by applying the change of basis
$$
\tau^{1} = \sigma^1, \ \quad
\tau^{2} = \sigma^2, \ \quad
\tau^{3} = \lambda \left( \sigma^3 + \frac{i\,\Real L_{\sigma}}{2\delta_{\sigma}}\,\sigma^1 \right),\ \quad
\tau^{4} = \lambda\,\sigma^4,
$$
where $\lambda\in\mathbb R^*$ is defined by
either $\lambda=1$ if $ \Imag L_{\sigma}=0$, or $\lambda=\frac{1}{\Imag L_{\sigma}}$ otherwise. 
The new structure equations still follow
Lemma~\ref{lema-FI-reducida1}, but now with coefficients $\varepsilon_{\tau}=\varepsilon_{\sigma}$,
$G_{\tau}=\lambda\,G_{\sigma}$, $b_{\tau}=\lambda\,b_{\sigma}$, $\delta_{\tau}=\delta_{\sigma}$
and $L_{\tau}=\lambda\,(L_{\sigma}-\Real L_{\sigma}) = i\,\lambda\,\Imag L_{\sigma} = i\nu\in\{0,i\}$,
in terms of the (1,0)-basis $\{\tau^{\,k}\}_{k=1}^4$.

Now, writing the complex coefficient $G_{\tau}$ as $G_{\tau}= |G_{\tau}| e^{i\alpha}$
for some $\alpha\in [0,2\,\pi)$, we define a new $(1,0)$-basis $\{\omega^i\}_{i=1}^4$ as follows:
$$\omega^{\,1} = e^{-\nicefrac{i\,\alpha}{2}}\,\tau^1, \quad
   \omega^{\,2} = \tau^2, \quad
   \omega^{\,3} = e^{-\nicefrac{i\,\alpha}{2}}\,\tau^3, \quad
   \omega^{\,4} = \tau^4.$$
This concludes the proof, simply denoting $a=|G_{\tau}|\geq 0$.
\end{proof}

After having reduced the complex structure equations of the Family~I,
next we  
study their equivalences in terms of the different parameters involved in our equations.

Let $J$ and $J'$ be two complex structures 
in Family I on an NLA $\frg$. Consider bases $\{\omega^{k}\}_{k=1}^4$ and 
$\{\omega'^{\,k}\}_{k=1}^4$ for $\frg_{J}^{1,0}$ and $\frg_{J'}^{1,0}$ 
satisfying structure equations as in Proposition~\ref{Prop-FI} with
parameters $(\varepsilon,\nu,\delta,a,b)$ and $(\varepsilon',\nu',\delta',a',b')$, respectively.
Any equivalence of complex structures, as presented at the beginning of
Section~\ref{complex-classif}, is defined by 
\begin{equation}\label{def-F-complejo}
F(\omega'^{\,i}) =\sum_{j=1}^4 \lambda_j^i\, \omega^j, \quad\text{for each } 1\leq i\leq 4,
\end{equation}
and satisfies the conditions
\begin{equation}\label{equivalence}
d\big(F(\omega'^{\,i})\big) = F(d\omega'^{\,i}),
\end{equation}
where the matrix $\Lambda=(\lambda^i_j)_{1\leq i,j\leq 4}$ belongs to ${\rm GL}(4,\mathbb{C})$.
To simplify our discussion, we will make use of the following notation.

\begin{notation}\label{Fij-complex}
We will denote by $\big[d\big(F(\omega'^{\,k})\big)-F(d\omega'^{\,k})\big]_{ij}$ the coefficient
for $\omega^{ij}$ in the expression $d\big(F(\omega'^{\,k})\big)-F(d\omega'^{\,k})$.
Similarly, for the coefficient of $\omega^{i\bar j}$.
\end{notation}

The following result reduces the general expression of the isomorphism~\eqref{def-F-complejo}.

\begin{lemma}\label{lema-previo-equiv-FI}
The forms $F(\omega'^{\,i})\in\frg^{1,0}_J$ satisfy the conditions:
$$
F(\omega'^{\,1})\wedge\omega^1=0, \quad \ 
    F(\omega'^{\,2})\wedge\omega^{12}=0, \quad \ 
    F(\omega'^{\,3})\wedge\omega^{123}=0, \quad \ 
    F(\omega'^{\,4})\wedge\omega^{14}=0.
$$
In particular, the matrix $\Lambda=(\lambda^i_j)_{1\leq i,j\leq 4}$ defining $F$ is triangular, 
and thus
$$
\Pi_{i=1}^4\lambda^i_i=\det\,\Lambda\neq 0.
$$
\end{lemma}

\begin{proof}
A direct calculation of the conditions \eqref{equivalence} for $i=1,2$ shows that
$\lambda^1_3 = \lambda^1_4 = \lambda^2_3 = \lambda^2_4 = 0$. 
Consequently, $F(\omega'^{\,1}), F(\omega'^{\,2})\in\langle \omega^1,\omega^2\rangle$.
In particular, $F(\omega'^{\,2})\wedge\omega^{12}=0$, as stated in the lemma.
Moreover, since $\Lambda$ 
becomes a block triangular matrix, we get 
$\text{det}\,\Lambda=\text{det}\,(\lambda^i_j)_{i,j=1,2}\cdot \text{det}\,(\lambda^i_j)_{i,j=3,4}\neq 0$.

If we now compute \eqref{equivalence} for $i=3$, then one in particular obtains 
$$\big[d\big(F(\omega'^{\,3})\big)-F(d\omega'^{\,3})\big]_{3\bar1} \, =  - i\,\delta\,\lambda^3_4=0,$$
which implies $\lambda^3_4=0$.
Thus, $F(\omega'^{\,3})\wedge\omega^{123}=0$ as required. Moreover, $0\neq \text{det}\,(\lambda^i_j)_{i,j=3,4} = \lambda^3_3\,\lambda^4_4$
and necessarily $\lambda^1_1\neq 0$, since these three coefficients are related by the
annihilation of
$$\big[d\big(F(\omega'^{\,3})\big)-F(d\omega'^{\,3})\big]_{14} \ = \ \lambda^3_3 - \lambda^1_1\,\lambda^4_4.$$
As a consequence of $\lambda^1_1,\,\lambda^3_3,\,\lambda^4_4$ being non-zero, the annihilation of
$$\big[d\big(F(\omega'^{\,3})\big)-F(d\omega'^{\,3})\big]_{13} \ = \ - \lambda^1_1\,\lambda^4_3,
\qquad
\big[d\big(F(\omega'^{\,3})\big)-F(d\omega'^{\,3})\big]_{24} \ = \ - \lambda^1_2\,\lambda^4_4$$
gives $\lambda^1_2=\lambda^4_3=0$.
In particular, we conclude that $F(\omega'^{\,1})\wedge\omega^1=0$.
Finally, from
$$0=\big[d\big(F(\omega'^{\,3})\big)-F(d\omega'^{\,3})\big]_{12} \ = \ -\lambda^1_1\,\lambda^4_2,$$
we obtain $\lambda^4_2=0$, i.e.  $F(\omega'^{\,4})\wedge\omega^{14}=0$.
\end{proof}

As a consequence, a first relation between the 
tuples $(\varepsilon,\nu,\delta,a,b)$ and $(\varepsilon',\nu',\delta',a',b')$ is attained:

\begin{proposition}\label{equiv-param-FI}
If the complex structures $J$ and $J'$ are equivalent, then
$$
\varepsilon' = \varepsilon,\quad \nu'=\nu,\quad \delta'=\delta.
$$
Moreover, there exists an isomorphism~\eqref{def-F-complejo}
satisfying the conditions in Lemma~\ref{lema-previo-equiv-FI} and
\begin{equation*}
\lambda^1_1 = e^{i\theta}, \qquad \lambda^4_4=\lambda\in\mathbb R^*, \qquad \lambda^3_3=\lambda\,e^{i\theta}, \qquad \nu(1-\lambda)=0, \qquad \varepsilon(1-\lambda^2_2)=0,
\end{equation*}
where $\theta\in[0,2\,\pi)$.
\end{proposition}

\begin{proof}
We first observe that Lemma~\ref{lema-previo-equiv-FI} must hold in order to have an equivalence between $J$ and $J'$ defined by $F$. Taking this as a starting point, let us recalculate the
conditions \eqref{equivalence} for each $1\leq i\leq 4$.
One can easily check that $F(d\omega'^{\,1}) = d\big(F(\omega'^{\,1})\big)$. 
For $i=2$, one simply has
\begin{equation}\label{condition-eps}
\begin{split}
0 = d\big(F(\omega'^{\,2})\big)-F(d\omega'^{\,2}) &= (\varepsilon\,\lambda^2_2 - \varepsilon'\,|\lambda^1_1|^2)\,\omega^{1\bar1}.
\end{split}
\end{equation}
If $\varepsilon=0$
then $\varepsilon'=0$, as Lemma~\ref{lema-previo-equiv-FI} states $\lambda^1_1\neq 0$.
Similarly, if $\varepsilon=1$ then $\lambda^2_2=\varepsilon'\,|\lambda^1_1|^2\neq 0$, and the only possibility
is taking $\varepsilon'=1$. These observations give 
$\varepsilon'=\varepsilon$.

For $i=3$, we highlight the following terms:
$$
\big[d\big(F(\omega'^{\,3})\big)-F(d\omega'^{\,3})\big]_{14} \ = \ \lambda^3_3 - \lambda^1_1\,\lambda^4_4, \qquad
\big[d\big(F(\omega'^{\,3})\big)-F(d\omega'^{\,3})\big]_{1\bar 4} \ = \ \lambda^3_3-\lambda^1_1\,\bar\lambda^4_4.
$$
Their annihilation leads to
\begin{equation}\label{coef-reales}
\lambda^4_4=\lambda\in\mathbb R^*, \qquad \lambda^3_3=\lambda\,\lambda^1_1.
\end{equation}

For $i=4$, one can take into account~\eqref{coef-reales} to get
$$
0 = \big[d\big(F(\omega'^{\,4})\big)-F(d\omega'^{\,4})\big]_{1\bar 3} \ = \  i\,\lambda\,(\delta-\delta'\,|\lambda^1_1|^2).
$$
Since $\lambda,\lambda^1_1\neq0$ and $\delta,\delta'\in\{-1,1\}$,
one necessarily has
$$\delta'=\delta,\quad |\lambda^1_1|^2=1.$$ In particular, we can set $\lambda^1_1= e^{i\,\theta}$, for some $\theta\in[0,2\pi)$.  Finally,
$$
0 = \big[d\big(F(\omega'^{\,4})\big)-F(d\omega'^{\,4})\big]_{1\bar 1} \ = \  i\,(\nu\,\lambda - \nu') - (b'\,|\lambda^2_1|^2 + 2\,\delta\,\Imag(\lambda^3_1\,e^{-i\theta})).
$$
The imaginary part of the previous equation implies that either $\nu'=\nu=0$
or $\nu'=\nu=1$ with $\lambda=1$, since $\lambda\neq 0$ and $\nu,\nu'\in\{0,1\}$.
Notice that this is equivalent to
$$\nu'=\nu,\quad \nu(1-\lambda)=0.$$
The expression $\varepsilon(1-\lambda^2_2)=0$ comes from rewriting~\eqref{condition-eps}.
\end{proof}

From now on, in order to determine the space of complex structures up to equivalence, we can
fix parameters $\varepsilon, \nu, \delta$ and simply identify $J$ and $J'$
with the pairs $(a,b)$ and $(a', b')$, respectively. Recall that $(a,b),\,(a',b')\neq (0,0)$ and $a,\,a'\geq 0$.

\begin{proposition}\label{equiv-a-b-FI}
The complex structures $J$ and $J'$ are equivalent if and only if there exists an isomorphism given by
\begin{equation}\label{equivalencia}
F(\omega'^{\,1}) = e^{i\theta}\,\omega^1,\quad F(\omega'^{\,2}) = \lambda^2_2\,\omega^2,\quad
F(\omega'^{\,3}) = \lambda\,e^{i\theta}\,\omega^3,\quad F(\omega'^{\,4}) = \lambda\,\omega^4,
\end{equation}
where $\theta\in[0,2\pi)$, $\lambda^2_2\in\mathbb C^*$, $\lambda\in\mathbb R^*$ and
\begin{equation}\label{equivalenciasI}
\Imag (\lambda_2^2\,e^{-2i\,\theta}) = 0,\quad \nu\,(1-\lambda) =0,\quad
\varepsilon\,(1 - \lambda^2_2 ) = 0.
\end{equation}
Moreover, the parameters $(a,b)$ and $(a', b')$ that respectively determine $J$ and $J'$
are related by
\begin{equation}\label{condiciones1}
a' = a\,\frac{\lambda}{\lambda^2_2\,e^{-2i\theta}},\qquad
b' = b\,\frac{\lambda}{|\lambda^2_2|^2}.
\end{equation}
\end{proposition}

\begin{proof}
According to the previous results, if $J$ and $J'$ are equivalent, then there exists an
isomorphism~$F$ defined by~\eqref{def-F-complejo} in the conditions of~Proposition~\ref{equiv-param-FI}.
We must ensure that $F$ fulfills \eqref{equivalence} for each $1\leq i\leq 4$.

First, one checks that the desired conditions are equivalent to the following equations:
\begin{equation*}
\begin{array}{lll}
0&\!\!=\!\!& \big[d\big(F(\omega'^{\,3})\big)-F(d\omega'^{\,3})\big]_{1\bar 1} \ = \ \,
	\varepsilon\,(\lambda^3_2 - i\,\delta\,b'\,\bar\lambda^2_1\,e^{i\,\theta})
	-\bar\lambda^4_1\,e^{i\theta} - a'\,\lambda^2_1\,e^{-i\theta},\\[5pt]
0&\!\!=\!\!& \big[d\big(F(\omega'^{\,3})\big)-F(d\omega'^{\,3})\big]_{1\bar 2} \ = \ \,
	i\,\delta\,\varepsilon\,e^{i\theta}\,(b\,\lambda - b'\,\bar\lambda^2_2),\\[5pt]
0&\!\!=\!\!& \big[d\big(F(\omega'^{\,3})\big)-F(d\omega'^{\,3})\big]_{2\bar 1} \ = \ \,
	a\,\lambda\,e^{i\theta} - a'\,\lambda^2_2\, e^{-i\theta},\\[5pt]
0&\!\!=\!\!& \big[d\big(F(\omega'^{\,4})\big)-F(d\omega'^{\,4})\big]_{1\bar 1} \ = \ \,
	-2\,\delta\,\Imag(\lambda^3_1\,e^{-i\theta})-b'\,|\lambda^2_1|^2,\\[5pt]
0&\!\!=\!\!& \big[d\big(F(\omega'^{\,4})\big)-F(d\omega'^{\,4})\big]_{1\bar 2} \ = \ \,
	-i\,\delta\,\bar\lambda^3_2\,e^{i\theta} - b'\,\lambda^2_1\,\bar\lambda^2_1,\\[5pt]
0&\!\!=\!\!& \big[d\big(F(\omega'^{\,4})\big)-F(d\omega'^{\,4})\big]_{2\bar 2} \ = \ \, b\,\lambda - b'\,|\lambda^2_2|^2.
\end{array}
\end{equation*}
Now, notice that the pairs $(a,b)$ and $(a', b')$,
which determine the complex structures $J$
and $J'$, are related by
$\big[d\big(F(\omega'^{\,3})\big)-F(d\omega'^{\,3})\big]_{2\bar 1}$ and
$\big[d\big(F(\omega'^{\,4})\big)-F(d\omega'^{\,4})\big]_{2\bar 2}$. Hence,
$a'$ and $b'$ are given by these expressions, obtaining
\eqref{condiciones1}.
In particular, the equivalence between $J$ and $J'$ only depends on
the parameters $\theta$, $\lambda$, and $\lambda^2_2$. Hence, the parameters
$\lambda^i_j$ for $i\neq j$ do not affect the relation \eqref{condiciones1}, and they can be
chosen to be zero. This solves the remaining equations and gives \eqref{equivalencia}.

Finally, the first expression in \eqref{equivalenciasI}
comes from imposing $a'\in\mathbb R$ in \eqref{condiciones1} whereas the other
two are a direct consequence of Proposition~\ref{equiv-param-FI}.
\end{proof}

Finally we can set the main result about equivalences of complex structures in Family I:

\begin{theorem}\label{equivalencias-familiaI}
Up to equivalence, the complex structures 
in Proposition~\ref{Prop-FI} 
%
%
%
%
are classified as follows:
$$
\begin{array}{rlrl}
{\rm \ (i)} & (\varepsilon, \nu, a, b) = (0,0,0,1),\, (0,0, 1, 0),\,   (0,0, 1, 1);
	&\quad {\rm (iii)} & (\varepsilon, \nu, a, b) = (1,0, 0,1),\,  (1,0, 1,|b|);\\[3pt]
{\rm (ii)} & (\varepsilon, \nu, a, b) = (0,1, 0, \nicefrac{b}{|b|}),\,  (0,1, 1,b);
	&\quad {\rm (iv)} & (\varepsilon, \nu, a, b) = (1,1, a, b).
\end{array}
$$

\end{theorem}

\begin{proof}
Let us study different cases depending on the values of the pair $(\varepsilon, \nu)$. Recall that
the conditions \eqref{equivalenciasI}--\eqref{condiciones1} given in Proposition~\ref{equiv-a-b-FI} must be satisfied,
namely
$$
a' = a\,\frac{\lambda}{\lambda^2_2\,e^{-2i\theta}},\qquad
b' = b\,\frac{\lambda}{|\lambda^2_2|^2},
$$
where
$\Imag (\lambda_2^2\,e^{-2i\,\theta}) = 0$, 
$\nu\,(1-\lambda) = 0$, and $\varepsilon\,(1 - \lambda^2_2 ) = 0$.
These expressions will give us the desired equivalences between $(a,b)$ and $(a',b')$, thus
between our complex structures.

\begin{itemize}
\item[(i)] $(\varepsilon, \nu)=(0,0)$:  There are no restrictions on $\lambda$ and $\lambda^2_2$, so
it is possible to normalize $a$ and/or $b$ (i.e. take $a'=1$ or $b'=1$) when they are non-zero.
Indeed, this is easy when $ab=0$, whereas for $ab\neq 0$ one can take $\theta=0$, $\lambda=\nicefrac{b}{a^2}$ and $\lambda^2_2=\nicefrac ba$.
\item[(ii)] $(\varepsilon, \nu)=(0,1)$:  In this case $\lambda=1$ and $\lambda^2_2\in\mathbb C^*$ is a free parameter.  When $a$ is non-zero, we can normalize it.  If $a=0$, we can choose $\lambda^2_2=\sqrt{|b|}$, and thus $b'=\pm1$.
\item[(iii)] $(\varepsilon, \nu)=(1,0)$:  From the expressions above we get $\lambda^2_2 = 1$, so $a' = a\lambda e^{2i\theta}$ and $b' = b\lambda$, for $\lambda\in\mathbb R^*$.
Observe that $e^{2i\theta}$ is a real number, so the only possible choices
are $\theta=0$ or $\theta=\nicefrac{\pi}{2}$.
If $a=0$ one can normalize $b'$, and if $a>0$ we can take $a'=1$ and $b'\geq 0$.  In fact,
in the last case it suffices to consider
$e^{2i\theta}=\nicefrac{b}{|b|}$
and $\lambda = \nicefrac{b}{a|b|}$.
\item[(iv)] $(\varepsilon, \nu)=(1,1)$:  We are forced to impose $\lambda = \lambda^2_2 = 1$, hence $a' = a e^{2i\theta}$ and $b' = b$.  Since $a,\,a'\geq 0$, necessarily $e^{2i\theta}=1$ and $a'=a$.\end{itemize}
\end{proof}

To complete the proof of Theorem~\ref{main-theorem} \textrm{(i)}, it remains to study the
ascending type of the Lie algebras underlying Family~I. 
For this we take as starting point the structure equations in Proposition~\ref{Prop-FI}.
Let $\{Z_k \}_{k=1}^4$ be the dual basis to $\{\omega^k\}_{k=1}^4$.
Then, a generic (real) element $X \in \frg$ can be written as 
\begin{equation}\label{vector-Z}
X = \sum_{i=1}^4 \alpha_i Z_i + \sum_{i=1}^4 \bar{\alpha}_i \bar Z_i,
\end{equation}
where $\alpha_i\in\mathbb C$, and $\bar Z_i$ is the conjugate of $Z_i$.
From the equations in Proposition~\ref{Prop-FI}, it follows that the brackets $[X, Z_k]$, for $1 \leq k \leq 4$,
are given by
\begin{equation}\label{corchetes-FI}
\begin{array}{l}
 [X, Z_1] =\varepsilon\, \bar{\alpha}_1 (Z_2-\bar Z_2) +\! (\alpha_4+\bar{\alpha}_4+i\,\delta\,\varepsilon\, b\,\bar{\alpha}_2) Z_3
- a\, \bar{\alpha}_2\, \bar Z_3+ i\,\nu\, \bar{\alpha}_1(Z_4+\bar Z_4) + i\,\delta\, \bar{\alpha}_3 (Z_4 - \bar Z_4),\\[5pt]
 [X, Z_2] = a\, \bar{\alpha}_1\, Z_3 +  i\,\delta\,\varepsilon\, b\,\bar{\alpha}_1\, \bar Z_3 + b\, \bar{\alpha}_2 (Z_4 - \bar Z_4),\\[5pt]
 [X, Z_3] = - \, i\,\delta\, \bar{\alpha}_1 (Z_4 - \bar Z_4),\\[5pt]
 [X, Z_4] = -\, \alpha_1\, Z_3  - \bar{\alpha}_1\, \bar Z_3.
\end{array}
\end{equation}
Note that since $X$ is real, the bracket $[X, \bar Z_k]$ is just the conjugate of $[X, Z_k]$, for $1 \leq k \leq 4$.
Recall that $\varepsilon,\nu\in\{0,1\}$, $\delta=\pm 1$, and $a,b\in\mathbb R$ with $a\geq 0$ and $(a,b)\neq (0,0)$.

Clearly, $[X, Z_4-\bar Z_4]=0$ for every $X\in\frg$,
so 
$\frg_1=\langle \Imag Z_4 \rangle$. 
Observe that this is consistent with \eqref{ReZ4}, as the change of basis in the proof of
Lemma~\ref{lema-FI-reducida1} switches the real and imaginary parts of~$\omega^4$, thus of~$Z_4$.

\begin{lemma}\label{lema-g2}
In the conditions above, the term $\frg_2$ in the ascending central series is given by:
\begin{itemize}
\item[(i)] for $a\neq0$: $\frg_2=\langle \Real Z_3,\Imag Z_3,\Imag Z_4 \rangle$
and so $\dim \frg_2 = 3$;
\smallskip
\item[(ii)] for $a=0$ and $\varepsilon=1$: $\frg_2=\langle 2\delta\,\Imag Z_2 + b\,\Real Z_4,\Real Z_3,\Imag Z_3,\Imag Z_4 \rangle$, hence $\dim \frg_2 = 4$;
\smallskip
\item[(iii)] for $a=\varepsilon=0$: 
$\frg_2=\langle \Real Z_2,\Imag Z_2,\Real Z_3,\Imag Z_3,\Imag Z_4 \rangle$
and so $\dim \frg_2 = 5$.
\end{itemize}
\end{lemma}

\begin{proof}
Let $X$ be a generic element in $\frg$ given by \eqref{vector-Z}.
Then, $X$ belongs to the term $\frg_2$ in the ascending central series if and only if
$[X, Z_k] \in\frg_1$, for every $1 \leq k \leq 4$.
Since $\mathfrak g_1=\langle \Imag Z_4\rangle$,
bearing in mind~\eqref{corchetes-FI} we get that $X\in\frg_2$ if and only if
$$
\alpha_1 = 0, \quad \  a\, \alpha_2 = 0, \quad \  \alpha_4+\bar{\alpha}_4 - i\,\delta\,\varepsilon\, b\,\alpha_2=0.
$$
In particular, one directly has $\Real Z_3,\, \Imag Z_3\in\frg_2$.
Now, to solve the previous system it suffices to distinguish the three different cases in the statement of the lemma. One gets the desired result simply substituting the corresponding solutions into~ \eqref{vector-Z}.
\end{proof}

\begin{proposition}\label{serie-familiaI}
Let $\frg$ be an $8$-dimensional NLA endowed with a complex structure $J$ in Family I with equations given in
Proposition~\ref{Prop-FI}. Then, the ascending type of $\frg$ is as follows:
\begin{itemize}
\item[(i)] if $a\neq 0$ and $(\varepsilon, \nu)
	\begin{cases}
	=(0,0), \text{ \ then \ }(\dim\frg_k)_k = (1,3,8);\\
	\neq (0,0), \text{ \ then \ }(\dim\frg_k)_k = (1,3,6,8);
	\end{cases}$

\smallskip
\item[(ii)] if $a=0$, $\varepsilon=1$ and $(\nu,b)
	\begin{cases}
	=(1,2\delta), \text{ \ then \ }(\dim\frg_k)_k = (1,4,8);\\
	\neq (1,2\delta), \text{ \ then \ }(\dim\frg_k)_k = (1,4,6,8);
	\end{cases}$

\smallskip
\item[(iii)] if $a=\varepsilon=0$ and $\nu=
	\begin{cases}
	0, \text{ \ then \ }(\dim\frg_k)_k = (1,5,8);\\
	1, \text{ \ then \ }(\dim\frg_k)_k = (1,5,6,8).
	\end{cases}$
\end{itemize}
%
\end{proposition}

\begin{proof}
A generic element $X$ given by \eqref{vector-Z}
belongs to the term $\frg_3$ in the ascending central series if and only if
$[X, Z_k] \in\frg_2$, for every $1 \leq k \leq 4$.
From~\eqref{corchetes-FI} it follows that this happens for $k=2,3$ and $4$, since
$\langle \Real Z_3,\, \Imag Z_3,\, \Imag Z_4 \rangle \subseteq \frg_2$
by Lemma~\ref{lema-g2}.
Hence, we must focus on the bracket $[X,Z_1]$, which can be
rewritten as
$$
[X, Z_1] = 2i\,\bar{\alpha}_1\left( \varepsilon\,\Imag Z_2 + \nu\,\Real Z_4 \right) + \Upsilon,
$$
for some $\Upsilon \in\frg_2$.
Therefore, $X$ will belong to $\frg_3$ depending on whether $\varepsilon\,\Imag Z_2 + \nu\,\Real Z_4 \in\mathfrak g_2$ or $\varepsilon\,\Imag Z_2 + \nu\,\Real Z_4 \notin\mathfrak g_2$. The analysis of these two cases leads to our result, bearing in mind the description of $\frg_2$ given in 
Lemma~\ref{lema-g2}.
\end{proof}

Combining the previous result with
Theorem~\ref{equivalencias-familiaI}, one obtains part \textrm{(i)} of Theorem~\ref{main-theorem}.

\subsection{Study of Family II}\label{reduc-fam-FII}	

In this section we arrive at the reduced equations~\eqref{FII-SnN} in
Theorem~\ref{main-theorem}, as well as at the classification of the complex structures in the Family II. Moreover, we study the ascending type of the $8$-dimensional nilpotent
Lie algebras admitting such complex structures, reaching Table~\ref{tab:valores-complejos2}.
Our starting point is Proposition~\ref{prop-resumen}.

\begin{lemma}\label{lema-previo-FII}
Let $J$ be a complex structure in Family II.
Then, there exists a basis of $(1,0)$-forms $\{\omega^k\}_{k=1}^4$ such that
\begin{equation}\label{ecusF2}
\begin{cases}
\begin{array}{ll}
d\omega^1 = 0, & \qquad  d\omega^3 = F\,\omega^{1\bar 1}  + C\,\omega^{12} + \bar C(\omega^{1\bar 2} -\omega^{2\bar1})-is\,(\omega^{24}+\omega^{2\bar 4}),\\
d\omega^2 = \omega^{14}+ \omega^{1\bar 4},& \qquad  d\omega^4 = L\,\omega^{1\bar 1}+s\,\omega^{2\bar 2}+i\,b\,(\omega^{1\bar 2}-\omega^{2\bar 1})
  +i\,(\omega^{1\bar 3}-\omega^{3\bar 1}),
\end{array}
\end{cases}
\end{equation}





\noindent where
the coefficients $C, F, L \in\mathbb C$ and $b, s\in\mathbb R$
are admissible; in particular, they satisfy 
$s\,\Imag L = 0$ and $(C, b, s)\neq (0,0,0)$.
\end{lemma}

\begin{proof}
It follows from Proposition~\ref{prop-resumen} that $J$ admits complex structure equations of the form \eqref{ecusG} with 
$B\neq 0$. By \eqref{jacobi-general} we also have $N\neq 0$, so one can define 
the $(1,0)$-basis
$$
\tau^1 = \omega^1,\quad\
\tau^2 = -\frac{i}{B}\omega^2,\quad\
\tau^3 = \bar N\omega^3-\frac{E\,\bar N}{B}\,\omega^2,\quad\
\tau^4 = i\,\omega^4 + i\,\frac{\bar A}{\bar B}\,\omega^1.
$$
With respect to $\{\tau^k\}_{k=1}^4$, the complex structure equations~\eqref{ecusG} in
Proposition~\ref{prop-resumen} become
$$
\begin{cases}
\begin{array}{ll}
d\tau^1 = 0, & \qquad  d\tau^3 = F_\tau\,\tau^{1\bar 1}  + C_\tau\,\tau^{12} +D_\tau\,\tau^{1\bar 2} +G_\tau\,\tau^{2\bar1}-H_\tau\,(\tau^{24}+\tau^{2\bar 4}),\\
d\tau^2 = \tau^{14}+ \tau^{1\bar 4},& \qquad  d\tau^4 = L_\tau\,\tau^{1\bar 1}+s_\tau\,\tau^{2\bar 2}+(M_\tau\,\tau^{1\bar 2}+\bar M_\tau\,\tau^{2\bar 1}) +i\,(\tau^{1\bar 3}-\tau^{3\bar 1}),
\end{array}
\end{cases}
$$



%

\noindent where the coefficients $F_\tau$, $C_\tau$, $D_\tau$, $G_\tau$, $H_\tau$, $L_\tau$, $M_\tau\in\mathbb C$ and
$s_\tau\in\mathbb R$ are expressed in terms of the original ones as follows:
$$
\begin{array}{lll}
    C_\tau = \frac{i\,B\,\bar N}{\bar B}\,\big( C\bar B - H\bar A\big),
    &\quad
    D_\tau = - i\,\bar B\,\bar ND,     &\quad
   G_\tau = i\,\bar N\big(BG-AH\big),\\[6pt]
    H_\tau = BH\bar N, &\quad  M_\tau = M\bar B + N\bar E, &\quad s_\tau = -s\,|B|^2.
\end{array}
$$
Note that, in order to get the result, it suffices to check that
$D_\tau = - G_\tau =\widebar{C}_\tau$, $H_\tau=i\,s_\tau$, and $M_\tau\in i\,\mathbb R$. This can be done using the conditions in~\eqref{jacobi-general}.
\end{proof}

\begin{lemma}\label{lema-FII-reducida-unif}
Let $J$ be a complex structure in Family II.
Then, there exists a basis of $(1,0)$-forms $\{\omega^k\}_{k=1}^4$ 
satisfying \eqref{ecusF2} 
with $C=\varepsilon$, $L=i\,\nu$ and $s=-\mu$, where 
$\varepsilon,\mu,\nu\in\{0,1\}$ such that $\mu\,\nu=0$ and $(\varepsilon,\mu,b)\neq(0,0,0)$. 
\end{lemma}

\begin{proof}
Since $J$ is in Family II, by Lemma~\ref{lema-previo-FII} there is a basis
$\{\omega^k\}_{k=1}^4$ satisfying the equations~\eqref{ecusF2}.
Hence, writing $C=|C|\,e^{i\,\beta}$ for some $\beta\in [0,2\,\pi)$, one can define
a new $(1,0)$-basis as follows:
$$
\tau^1=\lambda\,\omega^1, \,\,
\tau^2=\lambda\,\omega^2, \,\,
\tau^3=\frac{1}{\bar\lambda}\left( \omega^3 + \frac{i\,\Real L}{2}\,\omega^1 \right), \,\,
\tau^4=\omega^4,\quad\  \text{where}\ \ 
\lambda=
\begin{cases} 1,\qquad \qquad \, \text{ if } C=0,\\ |C|^{\nicefrac13}\,e^{i\,\beta},\quad \text{if }C\neq 0.
\end{cases}
$$
In terms of $\{\tau^k\}_{k=1}^4$ we obtain equations of the form~\eqref{ecusF2} with new coefficients
$C_{\tau}\in\{0,1\},\, F_\tau\in\mathbb C$, $L_\tau=i\,\nu_{\tau}$ for some $\nu_{\tau} \in \mathbb R$, and $b_\tau$, $s_\tau\in\mathbb R$. 
We rename $C_\tau=\varepsilon_\tau$ and $s_\tau = -\mu_\tau$.
Note that by Lemma~\ref{lema-previo-FII} we have $\mu_\tau\,\nu_\tau=0$ and $(\varepsilon_\tau, b_\tau, s_\tau)\neq (0,0,0)$.

Now, we focus our attention on the fact that the parameters $\mu_\tau,\nu_\tau\in\mathbb R$ 
satisfy the condition $\mu_\tau\,\nu_\tau=0$. Thus, we consider the following cases:

\noindent $\bullet$ Let us suppose $\mu_\tau=0$. If $\nu_\tau\neq 0$, then we can define the $(1,0)$-basis 
$\sigma^{\,1} = \tau^1$, $\sigma^{k} =\frac{1}{\nu_\tau}\,\tau^k$, for $k=2,3,4$, to get similar equations but with the normalized coefficient $\nu_\sigma=1$. 

\noindent $\bullet$ If $\mu_\tau\neq 0$, then $\nu_\tau=0$ and we consider
$\sigma^{\,1} = \tau^1$, $\sigma^{k} = \mu_\tau\,\tau^k$ for $k=2,3,4$. 
We arrive at similar equations but with the normalized coefficient $\mu_\sigma=1$.

Finally, renaming the basis and the coefficients, the lemma is proved. 
\end{proof}

Using the previous lemma, in the following result we arrive at the desired reduced structure equations~\eqref{FII-SnN}
of Theorem~\ref{main-theorem} for complex structures in the Family II.

\begin{proposition}\label{Prop-FII}
Every $8$-dimensional nilpotent Lie algebra $\frg$ endowed with a complex structure $J$
in Family II admits a basis of $(1,0)$-forms satisfying the structure equations
\begin{equation*}
\begin{cases}
\begin{array}{ll}
d\omega^1 = 0, & \qquad  d\omega^3 = a\,\omega^{1\bar 1}
 +\varepsilon\,(\omega^{12} + \omega^{1\bar 2} - \omega^{2\bar1}) + i\,\mu\,(\omega^{24} + \omega^{2\bar 4}),\\
d\omega^2 = \omega^{14}+ \omega^{1\bar 4},& \qquad  d\omega^4 = i\,\nu\,\omega^{1\bar 1} - \mu\,\omega^{2\bar 2} + i\,b\,(\omega^{1\bar 2} - \omega^{2\bar 1})
  + i\,(\omega^{1\bar 3} - \omega^{3\bar 1}),
\end{array}
\end{cases}
\end{equation*}


%
%

\noindent where $\varepsilon,\,\mu,\,\nu\in\{0,1\}$ such that $\mu\nu=0$, $(\varepsilon,\mu)\neq (0,0)$, and $a,b\in\mathbb R$.
\end{proposition}

\begin{proof}
By Lemma~\ref{lema-FII-reducida-unif}, 
it suffices to see that the coefficient $F$ in \eqref{ecusF2} can be chosen to be a real number.
Two cases are distinguished depending on $\varepsilon$:
\vskip.1cm

\noindent $\bullet$ If $\varepsilon=0$, it suffices to write $F=|F|\,e^{i\,\alpha}$, for some $\alpha\in[0,2\,\pi)$, and apply the change of basis defined by
$\tau^{k}=e^{-i\,\alpha}\,\omega^k$, for $k=1,2,3$, and $\tau^{\,4}=\omega^4$.

\vskip.1cm

\noindent $\bullet$ For $\varepsilon=1$, the result follows by considering the new $(1,0)$-basis defined by
$$
\tau^{1}=\omega^1, \quad
\tau^{\,2}=\omega^2 - \frac{i\,\Imag F}{2}\,\omega^1,\quad
\tau^{\,3}=\omega^3 + \frac{\mu\,\Imag F}{2}\,\omega^2 + \frac{i\,\Imag F\,(4\,b-3\,\mu\,\Imag F)}{8}\,\omega^1,\quad
\tau^{4}=\omega^4.
$$

Finally, $(\varepsilon, \mu)\neq(0,0)$ because otherwise the dimension
of the center of $\frg$ would be greater than 1. 
\end{proof}

After having reduced the complex structure equations of the Family~II,
we next 
study their equivalences in terms of the different parameters involved in the equations.
Let $J$ and $J'$ be two complex structures 
in Family II on an NLA $\frg$. Let $\{\omega^{k}\}_{k=1}^4$ and $\{\omega'^{\,k}\}_{k=1}^4$ be bases for 
$\frg_{J}^{1,0}$ and $\frg_{J'}^{1,0}$ 
satisfying structure equations as in Proposition~\ref{Prop-FII} with
parameters $(\varepsilon,\mu,\nu, a,b)$ and $(\varepsilon',\mu',\nu',a',b')$, respectively.
Any equivalence $F$ between $J$ and $J'$ is defined by \eqref{def-F-complejo}--\eqref{equivalence}.
Similarly to Lemma~\ref{lema-previo-equiv-FI}, for Family~II the isomorphism $F$ can be simplified as follows:

\begin{lemma}\label{lema-previo-equiv-FII}
The forms $F(\omega'^{\,i})\in\frg^{1,0}_J$ satisfy the conditions:
$$F(\omega'^{\,1})\wedge\omega^1=0, \quad
    F(\omega'^{\,2})\wedge\omega^{12}=0, \quad
    F(\omega'^{\,3})\wedge\omega^{123}=0, \quad
    F(\omega'^{\,4})\wedge\omega^{4}=0.$$
In particular, the matrix $\Lambda=(\lambda^i_j)_{1\leq i,j\leq 4}$ that defines $F$ is triangular
and thus
$$\Pi_{i=1}^4\lambda^i_i=\det\Lambda\neq 0.$$
\end{lemma}
\begin{proof}
The result comes straightforward by imposing $F(d\omega'^{\,i}) = d(F(\omega'^{\,i}))$ for $i=1,2,3$, taking into account that  $(\varepsilon,\,\mu) \neq (0,0)$ and $\text{det}\,\Lambda\neq 0$.
\end{proof}

Moreover, in a similar way to Proposition~\ref{equiv-param-FI} one finds a first relation between the tuples $(\varepsilon,\mu,\nu, a,b)$ and $(\varepsilon',\mu',\nu',a',b')$:

\begin{lemma}\label{lema-pre}
In the conditions above, if there exists an equivalence between $J$ and $J'$, then
$$
\varepsilon = \varepsilon',\quad \mu=\mu',\quad \nu=\nu'.$$
\end{lemma}

As a consequence, we can  focus on the relations between the pairs $(a,b)$ and $(a', b')$ to study the equivalences between the complex structures $J$ and $J'$. In fact, applying similar techniques as in the proof of Proposition~\ref{equiv-a-b-FI} one obtains the following result:

\begin{proposition}\label{suma1-FII}
Suppose that the complex structures $J$ and $J'$ are equivalent. We have:
\begin{enumerate}
\item[(i)] if $(\varepsilon,\mu,\nu)=(1,0,0)$, then $a' = \lambda\, a$ and $b'=b$, 
with $\lambda\in\mathbb R^*$;
\item[(ii)] if $(\varepsilon,\mu,\nu)=(0,1,0)$, then $a'=\nicefrac{a}{\kappa^5}$ and 
$b'= \nicefrac{1}{\kappa^2}\,\big( b + 2\,\kappa\,\Imag \lambda^2_1 \big)$, with $\kappa\in\mathbb R^*$;
\item[(iii)] otherwise, $a'=a$ and $b'=b$.
\end{enumerate}
Furthermore, any equivalence~$F$ between $J$ and $J'$ defined by~\eqref{def-F-complejo}--\eqref{equivalence} satisfies
Lemma~\ref{lema-previo-equiv-FII} together with
\begin{equation*}
\lambda^4_4=\lambda\in\mathbb R^*, \qquad
\lambda^2_2=\lambda\,\lambda^1_1, \qquad
\lambda^3_3=\frac{\lambda}{\bar\lambda^1_1}, \qquad
\lambda^3_2=i\,\mu\,\lambda\,\lambda^2_1,
\end{equation*}
and \
$\begin{cases}
\lambda^1_1=1, \quad \mathfrak{Im}\lambda^2_1=\mathfrak{Im}\lambda^3_1=0,
	\quad\text{if }(\varepsilon,\mu,\nu)=(1,0,0);\\[6pt]
\lambda^1_1=\kappa\in\mathbb R^*, \quad \lambda=\frac{1}{\kappa^2}, \quad
	\mathfrak{Im}\lambda^3_1=\frac{1}{\kappa}
		\left( \frac{1}{2}\,|\lambda^2_1|^2 - \frac{b}{\kappa}\,\mathfrak{Im}\lambda^2_1
			- 2\,(\mathfrak{Im}\lambda^2_1)^2\right),
	\quad\text{if }(\varepsilon,\mu,\nu)=(0,1,0).
\end{cases}$
 \end{proposition}

Finally, we arrive at the main result about equivalences of complex structures in Family II:

\begin{theorem}\label{equivalencias-familiaII}
Up to equivalence, the complex structures 
in Proposition~\ref{Prop-FII} are classified as follows:
$$\begin{array}{rlrl}
{\rm (i)} & (\varepsilon,\mu,\nu,a,b) = (1, 1, 0, a, b);
	&\qquad {\rm (iii)} & (\varepsilon,\mu,\nu,a,b) = (1, 0, 0, 0, b), (1, 0, 0, 1, b);\\
{\rm (ii)}& (\varepsilon,\mu,\nu,a,b) = (1, 0, 1, a, b);
	& \qquad {\rm (iv)} & (\varepsilon,\mu,\nu,a,b) = (0, 1, 0, 0, 0), (0, 1, 0, 1, 0).
\end{array}$$

%
%
%
\end{theorem}

\begin{proof}
We first observe that parts \textrm{(i)} and \textrm{(ii)} of the theorem come straightforward
from Proposition~\ref{suma1-FII}. Moreover, this proposition also gives us
the relation between the pairs $(a,b)$ and $(a',b')$ whenever there is an equivalence
between the complex structures $J$ and $J'$.

For $(\varepsilon,\mu,\nu)=(1, 0, 0)$, observe that one has
$a' = \lambda\,a$ and $b'=b$ with $\lambda\in\mathbb R^*$.
It suffices to set the values $\lambda^2_1=\lambda^3_1=0$ and, either 
$\lambda=\frac1a$ when $ a\neq 0$, or $\lambda=1$ when $a=0$, in order to obtain~\textrm{(iii)}.

For $(\varepsilon,\mu,\nu)=(0, 1, 0)$, we have
$a'=\frac{a}{\kappa^5}$ and 
$b'= \frac{1}{\kappa^2}\,\big( b + 2\,\kappa\,\Imag \lambda^2_1 \big)$,
with $\kappa\in\mathbb R^*$. 
Hence, one can take $\lambda^2_1 = -\frac{i\,b}{2\kappa}$, 
$\lambda_1^3 =\frac{i\,b^2}{8\,\kappa^3}$ and $(\lambda,\kappa) = (1,1)$ 
when $a=0$, or $(\lambda,\kappa) = (a^{-\nicefrac25}, a^{\nicefrac15})$ otherwise, 
so that we get $b'=0$ and $a'\in\{0,1\}$. This gives (iv) and completes the proof of the theorem.
\end{proof}

Finally, one can explicitly compute the ascending central series of the nilpotent Lie algebras underlying Family II from the equations given in Proposition~\ref{Prop-FII}. The ideas behind the proof are similar to those applied to Family I, so we omit the argument here.

\begin{proposition}\label{serie-familiaII}
Let $\frg$ be an $8$-dimensional NLA endowed with a complex structure $J$ in Family II with equations given in
Proposition~\ref{Prop-FII}. Then, the ascending type of $\frg$ is one of the following:
\begin{itemize}
\item[(i)] $(\dim \frg_k)_k = (1,3,5,8)$ if and only if $\nu =0$;
\item[(ii)] $(\dim \frg_k)_k = (1,3,5,6,8)$ if and only if $\nu =1$.
\end{itemize}
\end{proposition}

The combination of this result and Theorem~\ref{equivalencias-familiaII}
gives part \textrm{(ii)} of Theorem~\ref{main-theorem}.

It is worth noting that Propositions~\ref{serie-familiaI} and~\ref{serie-familiaII} show that the NLAs
underlying Families~I and~II do not share the same ascending type. This provides a stronger result than the partition of SnN complex structures 
into the families I and II proved in Proposition~\ref{dolb-0-1}.




\section{Classification of nilpotent Lie algebras with SnN complex structures}\label{clasif-real-FamI-II}

\noindent In this section we classify the $8$-dimensional NLAs admitting an SnN complex structure $J$. 
The study is divided into two parts depending on the family to which $J$ belongs 
(see Section~\ref{clasif-real-FamI} for $J$ in Family I and Section~\ref{clasif-real-FamII} for $J$ in Family II).
As a consequence, Theorem~\ref{intro-main-theorem-2} is proved.

%
%
\subsection{Classification of NLAs underlying Family I}\label{clasif-real-FamI}
The goal of this section is to prove that the non-isomorphic real Lie algebras underlying Family~I are those in Theorem~\ref{intro-main-theorem-2} denoted by
\begin{equation}\label{algebras-FI}
\mathfrak g_1^{\gamma}, \
\mathfrak g_2^{\alpha}, \
\mathfrak g_3^{\gamma}, \
\mathfrak g_4^{\alpha,\,\beta}, \
\mathfrak g_5, \
\mathfrak g_6, \
\mathfrak g_7, \
\mathfrak g_8,
\end{equation}
where $\gamma\in\{0,1\}$, $\alpha\in\mathbb R$, and $(\alpha, \beta)\in\mathbb R^*\times \mathbb R^{+}$ or $\mathbb R^{+}\times \{0\}$. Moreover, their ascending types are listed in the first column of Table~\ref{tab:cambios-FI} below.

\begin{notation}\label{notacion-lista-NLAs}
For the description of the structure equations of the Lie algebras in Theorem~\ref{intro-main-theorem-2} we are using the following abbreviated notation. For instance,
$$
\mathfrak g_1^{\gamma} = (0^5,\, 13+15+24,\, 14-23+25,\, 16+27+\gamma\cdot 34)
$$
means
that there is a basis $\{e^j\}_{j=1}^8$ of the dual $(\mathfrak g_1^{\gamma})^*$
of the nilpotent Lie algebra $\mathfrak g_1^{\gamma}$ satisfying $de^1 =\cdots = de^5 = 0$,
$de^6 = e^{1}\wedge e^{3} + e^{1}\wedge e^{5} + e^{2}\wedge e^{4}$,
$de^7 = e^{1}\wedge e^{4} - e^{2}\wedge e^{3} + e^{2}\wedge e^{5}$,
and
$de^8 = e^{1}\wedge e^{6} + e^{2}\wedge e^{7} + \gamma\, e^{3}\wedge e^{4}$.
\end{notation}

Let us recall that any $8$-dimensional NLA $\frg$ with complex structure $J$  in Family I is 
given by the complex structure equations~\eqref{FI-SnN} with parameters in 
Table~\ref{tab:valores-complejos}. For each tuple $(\varepsilon,\,\nu,\,a,\,b)$, it suffices to define a real basis $\{e^i\}_{i=1}^8$ following the third column in Table~\ref{tab:cambios-FI} to get the real Lie algebras above.

\begin{table}[p]
\centering
\renewcommand{\arraystretch}{1.4}
\renewcommand{\tabcolsep}{2pt}
\begin{tabular}{|c|c|l|c|}
\hline
 Ascending type & $(\varepsilon,\nu,a,b)$ & \ Real basis $\{e^k\}_{k=1}^8$ & NLA \\
\hline\hline
$(1,3,8)$ & $\begin{array}{c} (0,0,1,b),\\[-4pt] b\in\{0,1\} \end{array}$
& $\begin{array}{ll}
	\omega^1 = \delta\,e^1 - i\,e^2, & \omega^2 = - e^3 +\delta\, i\,e^4,\\
	\omega^3 =  \delta\,e^6 - i\,e^7, & \omega^4 = \frac12\,e^5 + 2\,\delta\, i\,e^8.\\
	\end{array}$
& $\mathfrak g_1^b$ \\
\hline\hline
\multirow{18}{*}{$(1,3,6,8)$} & $(0,1,1,b)$
&  $\begin{array}{ll}
	\omega^1 = \delta\,e^1 - i\,e^2, & \omega^2 = 4\,(\delta\,e^3 - i\,e^4),\\
	\omega^3 =  -4\,(e^6-\delta\,i\,e^7), & \omega^4 = -2\left(\delta\,e^5 + 4\,i\,e^8\right)\!.\\
	\end{array}$
& $\mathfrak g_2^{-4\delta b}$ \\
\cline{2-4}
	& $\begin{array}{c}(1,1,a,0)\\[-4pt] 0<a<2 \end{array}$
& \, (I)
& $\mathfrak g_2^{0}$ \\
\cline{2-4}
	& $(1,1,2,0)$
& $\begin{array}{l}
	\omega^1 = \frac{\delta}{\sqrt2}\,(e^1 + i\,e^2),\\
	\omega^2 = -\frac{1}{2}\,(e^3 - e^4) - i\,e^5, \\
	\omega^3 = \sqrt2 \,\delta\,(e^6 +i\,e^7), \\
	\omega^4 = \frac12\,(e^3 +e^4)  + e^5+ 2\,\delta\,i\,e^8.
	\end{array}$
& $\mathfrak g_3^{1}$ \\
\cline{2-4}
	& $\begin{array}{c}(1,1,a,0),\\[-4pt] a>2 \end{array}$
& \, (II)
& \multirow{3}{*}{$\mathfrak g_3^{0}$}\\
\cline{2-3}
	& $(1,0,1,0)$
& $\begin{array}{l}
	\omega^1 = \frac{1}{\sqrt2}\,(\delta\,e^1 - i\,e^2), \\
	\omega^2 = -\frac{1}{2}\,(e^3 - e^4) + i\,\delta\,e^5, \\
	\omega^3 = \frac{1}{\sqrt2}\,(\delta\,e^6 - i\,e^7), \ \ \omega^4 = \frac{1}{4}\,(e^3 +e^4) +i\,\delta\, e^8.\\
	\end{array}$\!\!\!
&  \\
\cline{2-4}
	& $\begin{array}{c}(1,0,1,b)\\[-4pt] b>0\end{array}$
& \multirow{3}{*}{\, (III)}
& \multirow{3}{*}{$\mathfrak g_4^{\frac {as}{b}, \frac{\vert b-2\delta \nu\vert}{a}}$} \\
\cline{2-2}
	& $\begin{array}{c}(1,1,a,b)\\[-4pt] a>0,\, b \neq  0,\, 2\delta\end{array}$
&
&  \\
\cline{2-4}
	& $\begin{array}{c}(1,1,a, 2\delta)\\[-4pt] a>0 \end{array}$
& $\begin{array}{l}
	\omega^1 = -\frac{1}{2}\left(e^1-e^2 - i\,(e^1+e^2)\right), \\
	\omega^2 = \frac{a}{2}\, e^4 + i\left(\frac{e^3}{2} + e^5\right), \\
	\omega^3 = -\frac{a}{2}\left(e^6-e^7 - i\,(e^6+e^7)\right), \\
	\omega^4 = -\left(\frac{a+2}{4}\, e^3  + e^5\right)+i\,\delta\,a\, e^8.
	\end{array}$
& $\mathfrak g_4^{\frac a2, 0}$ \\
\hline\hline
$(1,4,8)$ & $(1,1,0,2\delta)$
& $\begin{array}{ll}
	\omega^1 = \delta\,e^1 - i\,e^2, & \omega^2 = -\frac\delta 2\,(e^3 - 2\,i\,e^5), \\
	\omega^3 = \delta\,e^6 - i\,e^7, & \omega^4 = 2\delta\left(\frac\delta 4\,e^4 - \frac12\,e^5 +  i\,e^8\right).\\
	\end{array}$
& $\mathfrak g_5$ \\
\hline\hline
\multirow{4}{*}{$(1,4,6,8)$} & \multirow{2}{*}{$(1,0,0,1)$}
& \, $\omega^1 = \delta\,e^1 - i\,e^2,$
& \multirow{4}{*}{$\mathfrak g_6$} \\
	&
& \, $\omega^2 = (\frac{2\,\delta\,\nu}{b} - 1)\,e^3 + i\,\delta\,e^5,$
&	\\
\cline{2-2}
	& \multirow{2}{*}{$\begin{array}{c}(1,1,0,b)\\[-4pt] b\neq 0, 2\delta\end{array}$}
& \, $\omega^3 = (b-2\delta\,\nu) (\delta\,e^6 - i\,e^7),$
&	\\
	&
& \, $\omega^4 = \left(\frac b2-\delta\,\nu\right)\,e^4 - \delta\,\nu\,e^5 + 2\delta(b-2\,\delta\,\nu)\,i\,e^8.$
&	\\
\hline\hline
$(1,5,8)$ & $(0,0,0,1)$
& $\begin{array}{ll}
	\omega^1 = \delta\,e^1 - i\,e^2, & \omega^2 = -e^3 +\delta\,i\,e^4,\\
	\omega^3 = \delta\,e^6 - i\,e^7, & \omega^4 = \frac12\,e^5 + 2\,\delta\,i\,e^8.\\
	\end{array}$

& $\mathfrak g_7$ \\
\hline\hline
 $(1,5,6,8)$ & $\begin{array}{c}(0,1,0,b)\\[-4pt] b=\pm 1 \end{array}$
& $\begin{array}{ll}
	\omega^1 = \delta\,e^1 - i\,e^2, & \omega^2 = -(b\,e^3 +4\,i\,e^4),\\
	\omega^3 = -4\,(e^6 - i\,\delta\,e^7), & \omega^4 =-2\,(\delta\, e^5 + 4\,i\,e^8).\\
	\end{array}$
& $\mathfrak g_8$ \\
\hline
\end{tabular}
\medskip
\caption{Real Lie algebras and complex structures in Family I}
\label{tab:cambios-FI}
\end{table}

In Table~\ref{tab:cambios-FI}, note that $\delta=\pm1$ and $s=\sign(b-2\nu \delta)$. Moreover,
(I), (II), and (III) correspond to the following relations between the complex and the real bases:

\medskip
\noindent  (I) \ For $(\varepsilon,\nu,a,b)=(1,1,a,0)$ with $0<a<2$, one defines: 
\begin{equation*}
\begin{split}
\Real \omega^1 &= \frac{-\delta a^2}{2\sqrt3 (4-a^2)} \,e^2, \\
\Imag \omega^1 &= \frac{\delta a^2}{4\sqrt3 (4-a^2)}\left(\sqrt{4-a^2}\, e^1 + a e^2\right),\\
\Real \omega^2 &=  \frac{a^3}{24 (4-a^2)}\left(\sqrt{4-a^2}\, e^3 + a e^4\right),\\
\Imag \omega^2 &= \frac{-a^3}{12 (4\!-\!a^2)^{3/2}}\left(a e^3 \!-\! \sqrt{4\!-\!a^2} \, e^4 + a e^5\right),
\end{split}
\quad
\begin{split}
\Real \omega^3 &= \frac{\delta a^6}{48 \sqrt3 (4-a^2)^{2}}\left(a e^6 - \sqrt{4-a^2} \, e^7\right),\\
\Imag \omega^3 &= \frac{\delta a^6}{24 \sqrt3 (4-a^2)^{2}}\, e^6,\\
\Real \omega^4 &= \frac{a^4}{48 (4\!-\!a^2)^{3/2}}\left(a^2 e^3 \!-\! a\sqrt{4\!-\!a^2} \, e^4 + 4 e^5\right),\\
\Imag \omega^4 &= \frac{\delta a^8}{144 (4-a^2)^{5/2}}\, e^8.
\end{split}
\end{equation*}



\noindent (II) \ For $(\varepsilon,\nu,a,b)=(1,1,a,0)$ with $a>2$, one considers: 
\begin{equation*}
\begin{split}
\Real \omega^1 &=  -\frac{\delta a^2}{(a^2-4)} \sqrt{\frac32}\left(e^1 + \frac{1}{2-\sqrt3}\,e^2\right),
\\
\Imag \omega^1 &=  \frac{\delta a^2}{2(a^2-4)} \sqrt{\frac32}\left( (a+\sqrt{a^2-4})\,e^1+ \frac{a-\sqrt{a^2-4}}{2-\sqrt3}\,e^2\right),
\\
\Real \omega^2 &=  -\frac{3 a^3}{4 (a^2-4)}\left(\frac{a-\sqrt{a^2-4}}{(2-\sqrt3)^2}\,e^3 - (a+\sqrt{a^2-4})\,e^4\right),
\\
\Imag \omega^2 &= \frac{3 a^3}{2(a^2-4)^{3/2}} \left(\frac{ a-\sqrt{a^2-4}}{(2-\sqrt3)^2}\,e^3 + (a+\sqrt{a^2-4})\,e^4
	-\frac{2a}{2-\sqrt3}\,e^5\right),\\
\Real \omega^3 &= \frac{-\delta a^6}{(a^2-4)^2 (2-\sqrt3)} \left(\frac32\right)^{3/2} \left(\frac{a-\sqrt{a^2-4}}{2-\sqrt3}\,e^6
	- (a+\sqrt{a^2-4})\,e^7\right),
\\
\Imag\omega^3 &=\frac{-3\delta a^6}{(2-\sqrt3)(a^2-4)^2} \sqrt{\frac32}\left(\frac{1}{2-\sqrt3}\,e^6-e^7\right),
\end{split}
\end{equation*}
\begin{equation*}
\begin{split}
\Real \omega^4 &=  - \frac{3a^5}{8(a^2-4)^{3/2}} \left(\frac{a- \sqrt{a^2-4}}{(2-\sqrt3)^2}\,e^3 + (a+ \sqrt{a^2-4})\,e^4
	-\frac{8}{a (2-\sqrt 3)}\, e^5 \right),
\\
\Imag\omega^4 &=  \frac{-9\delta a^8}{(2-\sqrt3)^2(a^2-4)^{5/2}}\, e^8.
\end{split}
\end{equation*}


\bigskip
\noindent (III) \  For the cases $(\varepsilon,\nu,a,b)=(1,0,1,b)$ with $b>0$, and $(\varepsilon,\nu,a,b)=(1,1,a,b)$ with $a>0$ and $b\neq 0,2\delta$, let us define:
\begin{equation*}
\begin{split}
\Real\omega^1 &=-\frac{\delta}{2} \left(\sqrt{\frac{a}{a+s(b-2\nu \delta)}}\,e^1 -e^2\right),\\
\Imag\omega^1 &= \frac{s}{2} \left(\sqrt{\frac{a}{a+s(b-2\nu \delta)}}\,e^1+e^2\right), \\
\Real\omega^2 &= \frac{sa}{b}\,e^4, \\
\Imag\omega^2 &=  \delta\,s\, \sqrt{\frac{a}{a+s(b-2\nu \delta)}}\left(\frac{e^3}{2} + e^5\right),
\end{split}
\qquad
\begin{split}
\Real\omega^3 &=-\frac{a \delta}{2} \left(e^6- \sqrt{\frac{a}{a+s(b-2\nu \delta)}}\,e^7\right),\\
\Imag\omega^3 &= \frac{sa}{2} \left(e^6+\sqrt{\frac{a}{a+s(b-2\nu \delta)}}\,e^7\right), \\
\Real\omega^4 &=  -\sqrt{\frac{a}{a+s(b-2\nu \delta)}}\left(\frac{a+sb}{4}\,e^3 + \delta \nu s \,e^5\right), \\
\Imag\omega^4 &= \delta a\, \sqrt{\frac{a}{a+s(b-2\nu \delta)}}\,e^8.
\end{split}
\end{equation*}
\hskip.6cm

We now need to prove that the Lie algebras in \eqref{algebras-FI} are non-isomorphic. Obviously, this holds for NLAs having different ascending types. Hence, to complete the proof it suffices to analyze the NLAs underlying Family I within each of the different ascending types in Table \ref{tab:cambios-FI}.

The following invariants associated to NLAs will be relevant in our study:
\begin{itemize}
\item The descending type $(\text{dim}\,\frg^k)_k$: Recall that the \emph{descending central series} 
$\{\frg^k\}_{k\geq 0}$ of a Lie algebra~$\frg$ is defined by $\frg^0=\frg$ and
$\frg^k=[\frg^{k-1},\frg]$, for any $k\geq 1$. When $\frg$ is an $s$-step NLA, then $\frg_s=\{0\}$ and
we can associate an $s$-tuple to $\frg$, namely,
 $$(m^{1},\ldots,m^{s-1},m^s) :=\left(\text{dim\,}\frg^{1},\ldots, \text{dim\,}\frg^{s-1},\text{dim\,}\frg^{s}\right)$$
which strictly decreases, i.e. $2n>m^{1}>\cdots >m^{s-1}>m^{s}=0$.
We will say that $(\text{dim}\,\frg^k)_k=(m^{1},\ldots,m^{s})$ is the \emph{descending type} of~$\frg$.

\item The Betti numbers $b_k(\frg)$: The \emph{Chevalley-Eilenberg cohomology groups} of a Lie
algebra $\frg$ are defined by
$$H^k(\frg;\mathbb R)=\frac
   {\mathrm{Ker}\{d:\bigwedge^k\left(\frg^*\right)\longrightarrow\bigwedge^{k+1}\left(\frg^*\right)\}}
   {\mathrm{Im}\{d:\bigwedge^{k-1}\left(\frg^*\right)\longrightarrow\bigwedge^{k}\left(\frg^*\right)\}}, \qquad \text{ for }\quad 0\leq k\leq \dim \frg.$$
We will refer to their dimensions $b_{k}(\frg):=\text{dim}\,H^k(\frg;\mathbb{R})$ as the Betti numbers of $\frg$.

\item The number of functionally independent generalized Casimir operators $n_I(\frg)$: Let $\frg$
be an $m$-dimensional  Lie algebra with basis~$\{x_k\}_{k=1}^m$ and brackets
$[x_i,\,x_j]=\sum^m_{k=1} c_{ij}^k x_k$.
The vectors
\begin{equation}\label{vectores-Casimir}
\widehat X_k=\sum_{i,j=1}^m\,c^j_{ki}\,x_j\,\frac{\partial}{\partial\,x_i}
\end{equation}
generate a basis of the coadjoint representation of~$\frg$. One can construct a matrix~$C$ by rows
from the coefficients of these vectors, and then~$n_I(\frg)=m - \text{rank}\,C$.
For further details, see~\cite{SW}.
\end{itemize}

We start showing that the Lie algebras $\frg_1^0$ and $\frg_1^1$, which have ascending type $(1,3,8)$, are not isomorphic.
Although one can check that their descending types coincide,
the result comes as a direct consequence of their Casimir invariants. More precisely:

\begin{lemma}
Let $\frg_1^\gamma=(0^5,\, 13+15+24,\, 14-23+25,\, 16+27+\gamma\cdot 34)$, with $\gamma\in\{0,1\}$.
Then, $n_I(\frg_1^0)=4$ and $n_I(\frg_1^1)=2$. Therefore, $\frg_1^0$ and $\frg_1^1$ are not
isomorphic.
\end{lemma}

\begin{proof}
Using the equations of $\frak g_1^{\gamma}$ and the well-known formula~$de(X,Y)=-e([X,Y])$, for~$e\in\frg^*$ and~$X,Y\in\frg$, one can see that the matrix $C$ constructed from the coefficients of the vectors~\eqref{vectores-Casimir}~is
$$C=
\begin{pmatrix}
0 & 0 & -x_6 & -x_7 & -x_6 & -x_8 & 0 & 0 \\
0 & 0 &  x_7 & -x_6 & -x_7 & 0 & -x_8 & 0 \\
x_6 & -x_7 & 0 & -\gamma\,x_8 & 0 & 0 & 0 & 0 \\
x_7 &  x_6 & \gamma\,x_8 & 0 & 0 & 0 & 0 & 0 \\
x_6 &  x_7 & 0 & 0 & 0 & 0 & 0 & 0 \\
x_8 & 0 & 0 & 0 & 0 & 0 & 0 & 0 \\
0 & x_8 & 0 & 0 & 0 & 0 & 0 & 0 \\
0 & 0 & 0 & 0 & 0 & 0 & 0 & 0
\end{pmatrix}.$$
The minors of orders~$8$ and~$7$ are all equal to zero. For those of order~$6$, one obtains the
following expressions:
$$\gamma^2\,x_7^2\,x_8^4, \quad -\gamma^2\,x_6\,x_7\,x_8^4, \quad
   -\gamma^2\,x_7\,x_8^5, \quad \gamma^2\,x_6^2\,x_8^4, \quad
   \gamma^2\,x_6\,x_8^5, \quad \gamma^2\,x_8^6.$$
Consequently, if~$\gamma=1$ then rank\,$C=6$ and~$n_I=2$ for the algebra~$\mathfrak g_1$.
However, if~$\gamma=0$ then all the previous
expressions vanish, and it is possible to see that rank\,$C=4$, thus~$n_I=4$ for~$\mathfrak g_0$. This gives our result.
\end{proof}

It remains to study the NLAs with ascending type $(1,3,6,8)$.
In Table~\ref{Tabla-invariantes-FI}, we provide the descending type of the NLAs in the families
$\mathfrak g_2^{\alpha}$, $\mathfrak g_3^\gamma$, and $\mathfrak g_4^{\alpha,\beta}$,
as well as their number of functionally independent Casimir operators.
\begin{table}[h]
\begin{center}
\renewcommand{\arraystretch}{1.4}
\begin{tabular}{|c||c|c|c|c|}
\hline
$\frg$      & real parameter(s)    & $\dim \{\frg^k\}_k$     & $n_I(\frg)$ \\
\hline\hline
\multirow{2}{*}{$\frg_2^{\alpha}$} & $\alpha=0$ & \multirow{2}{*}{$(4,3,1,0)$} & $4$ \\ \cline{2-2}\cline{4-4}
            & $\alpha\neq 0$     &        &  $2$ \\
\hline\hline
\multirow{2}{*}{$\frg_3^{\gamma}$} & $\gamma=0$ & $(4,3,1,0)$ & \multirow{2}{*}{$4$} \\ \cline{2-3}
            & $\gamma=1$   & $(4,2,1,0)$ &   \\
\hline\hline
\multirow{3}{*}{$\frg_4^{\alpha,\beta}$} &  $\alpha\neq 0$, $\beta=1$ & $(4,2,1,0)$ & \multirow{3}{*}{$2$} \\ \cline{2-3}
            &  $\alpha\neq 0$, $\beta\in (0,1)\cup (1,\infty)$, & \multirow{2}{*}{$(4,3,1,0)$} & \\ \cline{2-2}
            &  $\alpha> 0$, $\beta=0$ &  &   \\
\hline
\end{tabular}
\vskip.1cm
\caption{Some invariants for the NLAs with ascending central series $(1,3,6,8)$}
\label{Tabla-invariantes-FI}
\end{center}
\end{table}

\vspace{-0.6cm}
A direct consequence of these invariants is that some of the NLAs in the previous three families cannot be isomorphic. 
In particular, it suffices to prove the result below:

\begin{proposition}\label{demostrar}
The following
pairs of Lie algebras are not isomorphic:
\begin{itemize}
\item[(i)] $\frak g_2^0$ and $\frak g_3^0$.
\item[(ii)] $\frak g_2^\alpha$ and $\frak g_2^{\alpha'}$ whenever $\alpha \neq \alpha'$ and $\alpha\alpha'\neq0$.
\item[(iii)] $\frak g_2^{\alpha'}$ and $\frak g_4^{\alpha, \beta}$ whenever $\alpha' \neq 0$ and $\beta\neq 1$.
\item[(iv)] $\frak g_4^{\alpha, 1}$ and $\frak g_4^{\alpha', 1}$ whenever $\alpha\neq \alpha'$.
\item[(v)] $\frak g_4^{\alpha, \beta}$ and $\frak g_4^{\alpha', \beta'}$ whenever $(\alpha, \beta)\neq (\alpha', \beta')$ and $\beta, \beta' \neq 1$.
\end{itemize}
\end{proposition}

In order to study the five cases above, one directly analyzes the existence of isomorphisms
between any two of the previous Lie algebras.
The procedure is quite similar to that used to prove the non-equivalence of complex structures in Section~\ref{reduc-fam-FI}. Indeed, 
let $f \colon \frg \longrightarrow \frg'$ be an homomorphism
of Lie algebras. Its dual map $f^* \colon \frg'^* \longrightarrow \frg^*$ naturally extends to a map
$F \colon \bigwedge^*\frg'^* \longrightarrow \bigwedge^*\frg^*$ that commutes with the differentials, i.e. $F\circ d=d\circ F$.
If $\{e^i\}_{i=1}^8$ and $\{e'^{\,i}\}_{i=1}^8$ are any bases for~$\frg^*$
and $\frg'^*$, respectively, then any Lie algebra isomorphism is defined by
\begin{equation}\label{cambio-base}
F(e'^{\,i}) =\sum_{j=1}^8 \lambda_j^i\, e^j, \quad i=1,\ldots,8,
\end{equation}
and satisfies the conditions
\begin{equation}\label{condiciones}
F(de'^{\,i}) = d(F(e'^{\,i})), \ \text{ for each }1\leq i\leq 8,
\end{equation}
where the matrix $\Lambda=(\lambda^i_j)_{1\leq i,j\leq 8}$ belongs to ${\rm GL}(8,\mathbb{R})$.
Taking bases for $\frg,\,\frg'\in\left\{ \frg_2^\alpha, \, \frg_3^0, \, \frg_4^{\alpha,\beta} \right\}$ satisfying the corresponding structure equations
in Theorem~\ref{intro-main-theorem-2},
one obtains a reduction of $F$ that eventually allows to prove the desired result.
As this is a very technical proof due to the several cases to be considered, we omit the details here. Nonetheless, we refer the reader to the proof of Proposition~\ref{no-iso-m4} in this paper to get an idea of how it works.

%
%
\subsection{Classification of NLAs underlying Family II}\label{clasif-real-FamII}
The goal of this section is to prove that the non-isomorphic $8$-dimensional NLAs that admit complex structures in the Family~II are those in Theorem~\ref{intro-main-theorem-2} denoted by
\begin{equation*}
\mathfrak g_{9}^{\gamma}, \
\mathfrak g_{10}^{\gamma}, \
\mathfrak g_{11}^{\alpha, \beta}, \
\mathfrak g_{12}^{\gamma},
\end{equation*}
where $\gamma\in\{0,1\}$ and $(\alpha, \beta)=(0,0), (1,0),(0,1)$ or $(\alpha, 1)$ with $\alpha\in\mathbb R^{+}$. Moreover,
their ascending types are listed in the first column of Table~\ref{tab:cambios-FII}.

Let $(\frg,J)$ be an $8$-dimensional NLA endowed with a complex structure. If $J$ belongs to 
Family II, then the complex structure equations of $(\frg,J)$ are given by~\eqref{FII-SnN} with 
parameters in 
Table~\ref{tab:valores-complejos2}.
For each tuple $(\varepsilon, \mu,\nu, a,b)$, define a real basis $\{e^i\}_{i=1}^8$ according to the third column in Table~\ref{tab:cambios-FII} to find the real Lie algebras above. 
In the table, the following notation is used
$$\tau = \left\{\begin{array}{ll}
1, & a\geq 0,\\[5pt]
-1, & a<0,
\end{array}\right.
\text{ \ \ and \ \ }
\eta = \left\{\begin{array}{ll}
1,& b=0,\\[5pt]
\dfrac{-\sqrt3}{4b},& b\neq 0.
\end{array}\right.
$$

\begin{table}[ht]
\centering
\renewcommand{\arraystretch}{1.4}
\renewcommand{\tabcolsep}{2pt}
\begin{tabular}{|c|c|l|c|}
\hline
\begin{tabular}{c} Ascending \\[-8pt] type \end{tabular} & $(\varepsilon,\mu,\nu,a,b)$ & \ Real basis $\{e^k\}_{k=1}^8$ & NLA \\
\hline\hline
\multirow{18}{*}{$(1,3,5,8)$} & \multirow{2}{*}{$(0,1,0, 0,0)$}
& \, $\omega^1 = -\frac18\,(e^1 + i\,e^2),$
& \multirow{2}{*}{$\mathfrak g_{9}^0$} \\
	&
& \, $\omega^2 = \frac1{16}\,(e^4 + i\,e^5),$
&	\\
\cline{2-2}\cline{4-4}
	& \multirow{2}{*}{$(0,1,0, 1, 0)$}
& \, $\omega^3 = -\frac1{32}\,(e^6 + i\,e^7),$
& \multirow{2}{*}{$\mathfrak g_{9}^1$} \\
	&
& \, $\omega^4 = -\frac1{128}\,(32\,e^3 - i\,e^8).$
&	\\
\cline{2-4}
	& \multirow{2}{*}{$\begin{array}{c}(1,0,0, a, 0)\\[-4pt] a\in\{0,1\}\end{array}$}
& \, $\omega^1 = \frac1{2\eta}\left(-\frac{e^1}{\sqrt3} + i\,e^2\right),$
& \multirow{2}{*}{$\mathfrak g_{10}^0$} \\
	& 
& \, $\omega^2 = \frac1{2\eta^3}\left(\frac{e^4}{3} - i \left(\frac{e^5}{\sqrt3} - 2\,\gamma\,\eta^2 e^2\right)\right),$
&	\\
\cline{2-2}\cline{4-4}
	& \multirow{2}{*}{$\begin{array}{c}(1,0,0, a, b)\\[-4pt] a\in\{0,1\},\, b\neq 0\end{array}$}
& \, $\omega^3 = \frac1{12\eta^4}\,(-\sqrt3\,e^6 + i\,e^7),$
& \multirow{2}{*}{$\mathfrak g_{10}^1$} \\
	&
& \, $\omega^4 = \frac1{6\eta^5}\left(-\sqrt3\,\eta^3\,e^3 + i \left(\frac{e^8}{2} - \frac{b \eta}{\sqrt3}\,e^6\right)\right).$
&	\\
\cline{2-4}
	& $(1,1,0, 0, 0)$
& $\begin{array}{l}
	\omega^1 = \frac1{2}\left(-e^1+ i\,\sqrt3\,e^2\right), \\
	\omega^2 = \frac{\sqrt3}{4}\left(e^4- i\,\sqrt3\,e^5\right), \\
	\omega^3 = \frac38(-\sqrt3\,e^6 + i\,e^7), \\
	\omega^4 = \frac{\sqrt3}{4}\left(- e^3 + \frac{3i}{2}\,e^8\right).
	\end{array}$
& $\mathfrak g_{11}^{0,0}$ \\
\cline{2-4}
	& $\begin{array}{c}(1,1,0, a, 0)\\[-4pt] a\neq 0\end{array}$
& $\begin{array}{l}
	\omega^1 = \frac {2a}{\sqrt3}\left(-e^1+ i\,\sqrt3\,e^2\right), \\
	\omega^2 = \frac{4a^2} {\sqrt3}\left(e^4- i\,\sqrt3\,e^5\right), \\
	\omega^3 = \frac{8 a^3}{\sqrt3}(-\sqrt3\,e^6 + i\,e^7), \\
	\omega^4 =a\left(- e^3 + \frac{32\, i\, a^3}{\sqrt3}\,e^8\right).
	\end{array}$
& $\mathfrak g_{11}^{1,0}$ \\
\cline{2-4}
	& $\begin{array}{c}(1,1,0, a, b)\\[-4pt] b\neq 0\end{array}$
& $\begin{array}{l}  
	\omega^1 = -\frac {b}{3}\left(e^1- \frac{i\,\sqrt3\,\tau\,b}{|b|}\,e^2\right), \\
	\omega^2 = \frac{b^2} {3\,\sqrt3}\left(\frac{\tau\,b}{|b|}\,e^4- i\,\sqrt3\,e^5\right), \\
	\omega^3 = -\frac{b^3}9\,\left(\frac{\sqrt3\,\tau\,b}{|b|}\, e^6 - i\,e^7\right), \\
	\omega^4 = -\frac{\tau\,|b|}{2\sqrt3}\left(e^3 + \frac{4\,i\,b^3}{9} \left(e^6 -e^8\right)\right).   
	\end{array}$
& $\mathfrak g_{11}^{\frac{2\sqrt3|a|}{|b|}, 1}$\\
\hline\hline
\multirow{4}{*}{$(1,3,5,6,8)$} & \multirow{2}{*}{$(1,0,1, a, 0)$}
& \, $\omega^1 = \frac1{2\eta}\left(-\frac{e^1}{\sqrt3} + i\,e^2\right),$
& \multirow{2}{*}{$\mathfrak g_{12}^0$} \\
	&
& \, $\omega^2 = \frac1{2\eta^3}\left(\frac{e^4}{3} - i \left(\frac{e^5}{\sqrt3} - 2\,a\,\eta^2 e^2\right)\right), $
&	\\
\cline{2-2}\cline{4-4}
	& \multirow{2}{*}{$\begin{array}{c}(1,0,1, a, b)\\[-4pt] b\neq 0\end{array}$}
& \, $\omega^3 = \frac1{12\eta^4}(-\sqrt3\,e^6 + i\,e^7),$
& \multirow{2}{*}{$\mathfrak g_{12}^1$} \\
	&
& \, $\omega^4 = \frac1{6\eta^5}\left(-\sqrt3\,\eta^3\,e^3 + i \left(\frac{e^8}{2} - \frac{b \eta}{\sqrt3}\,e^6\right)\right).$
&	\\
\hline
\end{tabular}
\medskip
\caption{Real Lie algebras and complex structures in Family II}
\label{tab:cambios-FII}
\end{table}

One can easily check that the Lie algebras $\mathfrak g_{9}^{0}$ and $\mathfrak g_{9}^{1}$ are not isomorphic, as the former has four decomposable $d$-exact $2$-forms, while the latter only has three.

Note also that the second Betti numbers of $\mathfrak g_{10}^{0}$ and $\mathfrak g_{10}^{1}$ do not coincide, as

\vskip.2cm

\hskip2cm $H^2 (\mathfrak g_{10}^0) =
    \langle\, [e^{12}],\,[e^{25}],\,[e^{34}],\,[e^{35}],\,[e^{17}+e^{26}],\,[e^{38}-e^{46}-e^{57}] \,\rangle$,

\vskip.2cm

\hskip2cm  $H^2 (\mathfrak g_{10}^1) =
    \langle\, [e^{12}],\,[e^{25}],\,[e^{34}],\,[e^{35}],\,[e^{17}+e^{26}] \,\rangle$.

\vskip.2cm
    
Therefore, these two NLAs are not isomorphic.

The real Lie algebras~$\mathfrak g_{11}^{\alpha, \beta}$ are studied by the authors in~\cite{LUV-complex}, where it is proved that they are non-isomorphic for different values of $\alpha,\beta$ satisfying $(\alpha, \beta)=(0,0), (1,0)$ or $(\alpha^{\geq 0}, 1)$.

\bigskip

To finish the study for the ascending type $(1,3,5,8)$ one needs to prove that there are no isomorphisms between $\mathfrak g_{i}^{\bullet}$ and $\mathfrak g_{j}^{\bullet}$ for $i,j\in\{9,10,11\}$, $i\neq j$. Although the descending type of these three families is exactly the same, namely~$(5,3,1,0)$, one can
make use of the number of functionally independent Casimir invariants $n_I$
and the second Betti number $b_2$. In fact,  these two invariants allow us to conclude that there are no isomorphisms
between any two
NLAs belonging to two different families (see Table~\ref{Tabla-invariantes-FII}).

\begin{table}[h!]
\begin{center}
\renewcommand{\arraystretch}{1.4}
\begin{tabular}{|c||c|c|c|c|}
\hline
$\frg$      & real parameter(s)    & $b_2(\frg)$     & $n_I(\frg)$ \\
\hline\hline
$\mathfrak g_{9}^{\gamma}$ & $\gamma\in\{0,1\}$ & $6$ & $2$ \\
\hline\hline
\multirow{2}{*}{$\mathfrak g_{10}^{\gamma}$} & $\gamma=0$ & $6$ & \multirow{2}{*}{$4$} \\ \cline{2-3}
            & $\gamma=1$   & $5$ &   \\
\hline\hline
$\mathfrak g_{11}^{\alpha,\beta}$ &  $(\alpha,\beta)=(0,0),(1,0)$, or $(\alpha^{\geq 0}, 1)$ & $4$ & $2$ \\
\hline
\end{tabular}
\vskip.1cm
\caption{Some invariants for the NLAs with ascending type $(1,3,5,8)$}
\label{Tabla-invariantes-FII}
\end{center}
\end{table}

One now needs to show that $\mathfrak g_{12}^{0}$ and $\mathfrak g_{12}^{1}$ are not isomorphic.
For this purpose we make use of the following result:

\begin{lemma}\label{lema-F2}
Let $\frg=\mathfrak g_{12}^{0}$ and $\frg'=\mathfrak g_{12}^{1}$. Let $\{e^i\}_{i=1}^8$ and $\{e'^{\,i}\}_{i=1}^8$
be respective bases for $\frg^*$ and $\frg'^*$ satisfying the corresponding structure equations in Theorem~\ref{intro-main-theorem-2}.  If
$f\colon\frg\longrightarrow\frg'$ is an isomorphism of Lie algebras, then the dual map
$f^*\colon\frg'^{\,*}\longrightarrow\frg^*$ satisfies
\begin{equation}\label{cambio-general-reducido}
\begin{split}
&f^*(e'^{\,i}) \wedge e^{12} =0, \text{ \ for } i=1,2,\\
&f^*(e'^{\,3}) \wedge e^{123} =0,
\end{split}
\qquad\quad
\begin{split}
&f^*(e'^{\,i}) \wedge e^{12345} =0,  \text{ for } i=4,5.\\
&f^*(e'^{\,i}) \wedge e^{1234567} =0,  \text{ for } i=6,7.
\end{split}
\end{equation}
\end{lemma}

\begin{proof}
More generally, let $f\colon\frg\longrightarrow\frg'$ be an isomorphism of two $m$-dimensional Lie algebras.
Consider an ideal $\{0\}\neq \fra \subset \frg$, and let $\fra'=f(\fra) \subset \frg'$
be the corresponding ideal in $\frg'$. Let $\{x_{r+1},\ldots,x_m\}$ and $\{x'_{r+1},\ldots,x'_m\}$ be any bases for~$\fra$ and~$\fra'$, respectively.
We complete them up to respective bases $\{x_1,\ldots,x_r,x_{r+1},\ldots,x_m\}$ and $\{x'_1,\ldots,x'_r,x'_{r+1},\ldots,x'_m\}$ for~$\frg$ and~$\frg'$.
Denote the dual bases of~$\frg^*$ and~$\frg'^{\,*}$, respectively, by $\{x^i\}_{i=1}^m$
and $\{x'^{\,i}\}_{i=1}^m$. In these conditions, it is proved in \cite[Lemma 4.2]{LUV-gen-complex} that the dual map $f^*\colon\frg'^{\,*}\longrightarrow\frg^*$
satisfies
\begin{equation}\label{tech}
f^*(x'^{\,i}) \wedge x^{1} \wedge \ldots \wedge x^{r}=0, \qquad \mbox{ for all }\ i=1,\ldots,r.
\end{equation}

We will apply this result to different choices of $\fra$.
Denote by $\{e_i\}_{i=1}^8$ the basis for $\frg=\frg_{12}^0$ dual to a basis
$\{e^i\}_{i=1}^8$ of $\frg^*$ that satisfies the corresponding structure equations in 
Theorem~\ref{intro-main-theorem-2}. Proceed similarly to define $\{e'_i\}_{i=1}^8$ for $\frg'=\frg_{12}^1$ and $\{e'^{\,i}\}_{i=1}^8$ for $\frg'^{\,*}$.
The ascending central series of $\frg$ and $\frg'$ are
$$
\mathfrak{g}_1=\langle e_8 \rangle \subset
\mathfrak{g}_2=\langle e_6,e_7,e_8 \rangle \subset
\mathfrak{g}_3=\langle e_4,e_5,e_6,e_7,e_8 \rangle \subset
\mathfrak{g}_4=\langle e_3,e_4,e_5,e_6,e_7,e_8 \rangle,
$$
and
$$
\mathfrak{g'}_1=\langle e'_8 \rangle \subset
\mathfrak{g'}_2=\langle e'_6,e'_7,e'_8 \rangle \subset
\mathfrak{g'}_3=\langle e'_4,e'_5,e'_6,e'_7,e'_8 \rangle \subset
\mathfrak{g'}_4=\langle e'_3,e'_4,e'_5,e'_6,e'_7,e'_8 \rangle.
$$
Since $f(\mathfrak{g}_k)=\mathfrak{g}'_k$ for any Lie algebra isomorphism $f\colon\frg\longrightarrow\frg'$,
by \eqref{tech} applied to $\mathfrak{a}=\frg_k$ for $1\leq k\leq 4$  
one finds~\eqref{cambio-general-reducido}.
\end{proof}

\smallskip
This allows us to prove the following proposition, which completes our result.

\begin{proposition}\label{no-iso-m4}
The Lie algebras $\mathfrak g_{12}^{0}$ and $\mathfrak g_{12}^{1}$ are not isomorphic.
\end{proposition}

\begin{proof}
Let $\{e^i\}_{i=1}^8$ and $\{e'^{\,i}\}_{i=1}^8$ be bases of $(\mathfrak g_{12}^{0})^*$ and $(\mathfrak g_{12}^{1})^*$,
respectively, satisfying the corresponding structure equations in Theorem~\ref{intro-main-theorem-2}.
Consider a Lie algebra isomorphism $F$ between $\mathfrak g_{12}^{0}$ and $\mathfrak g_{12}^{1}$ given by~\eqref{cambio-base}.
In what follows, and similarly to Notation~\ref{Fij-complex}, we will denote by $\big[d\big(F(e'^{\,k})\big)-F(de'^{\,k})\big]_{ij}$ the coefficient for $e^{ij}$ in the 2-form $d\big(F(e'^{\,k})\big)-F(de'^{\,k})$.
Observe that
the conditions~\eqref{condiciones} are equivalent to
$$\big[d\big(F(e'^{\,k})\big)-F(de'^{\,k})\big]_{ij}=0, \text{ for every }1\leq k\leq  8  \text{ and }1\leq i<j\leq 8.$$

We first notice that, as a consequence of Lemma~\ref{lema-F2}, one has
\begin{equation*}
\lambda^i_j =0, \text{ for }
	\begin{cases}
	1\leq i \leq 2 \text{ and } 3\leq j \leq 8,\\
	3\leq i \leq 5 \text{ and } 6\leq j \leq 8,
	\end{cases}
\text{ and \ also \ }
\lambda^3_4 = \lambda^3_5 = \lambda^6_8 =\lambda^7_8 =0.
\end{equation*}
Moreover, $\lambda^3_3\lambda^8_8\neq 0$ in order to ensure 
$\Lambda=(\lambda^i_j)_{1\leq i,j\leq 8}\in{\rm GL}(8,\mathbb{R})$.
We next show that:
\begin{equation}\label{L11}
\lambda^2_2 = \epsilon\, \lambda^1_1,\quad \lambda^2_1 = \epsilon\, \lambda^1_2,\quad\text{where } \epsilon=\pm 1.
\end{equation}

Let us observe that $[d(F(e'^{\,3}))-F(de'^{\,3})]_{12}=0$ together with $[d(F(e'^{\,k}))-F(de'^{\,k})]_{13}=0$ and $[d(F(e'^{\,k}))-F(de'^{\,k})]_{23}=0$, for $k=4,5$, give the following expressions:
\begin{equation}\label{de-interes-1}
\lambda^3_3 = \lambda^1_1\lambda^2_2 - \lambda^1_2\lambda^2_1 \ \ (\neq 0),\qquad
\lambda^4_4 = \lambda^1_1\lambda^3_3, \qquad
\lambda^4_5 = \lambda^1_2\lambda^3_3, \qquad
\lambda^5_4 = \lambda^2_1\lambda^3_3, \qquad
\lambda^5_5 = \lambda^2_2\lambda^3_3.
\end{equation}
Furthermore, thanks to $[d(F(e'^{\,k}))-F(de'^{\,k})]_{14}=0$ for $k=6,7$ one can solve
\begin{equation}\label{L66}
\lambda^6_6=\lambda^1_1\lambda^4_4 + \lambda^2_1\lambda^5_4, \qquad
\lambda^7_6 = \lambda^1_1\lambda^5_4 + \lambda^2_1\lambda^4_4.
\end{equation}
Considering these values and those in~\eqref{de-interes-1}, from
$[d(F(e'^{\,k}))-F(de'^{\,k})]_{25}=0$, with $k=6,7$, one gets the system of equations 
\begin{equation}\label{12}
(\lambda^1_1)^2 - (\lambda^1_2)^2 + (\lambda^2_1)^2 - (\lambda^2_2)^2=0,\qquad
\lambda^1_1\lambda^2_1 - \lambda^1_2\lambda^2_2=0.
\end{equation}
If we suppose $\lambda^1_1=0$, then~\eqref{L11} follows directly as a consequence of $\lambda^3_3=- \lambda^1_2\lambda^2_1\neq 0$.  Otherwise, one can solve $\lambda^2_1 = \frac{\lambda^1_2 \lambda^2_2}{\lambda^1_1}$
using the second equation in~\eqref{12}, and then
$$\lambda^3_3 = \frac{\lambda^2_2}{\lambda^1_1} \left((\lambda^1_1)^2 - (\lambda^1_2)^2\right).$$ In particular, one observes
that $\lambda^2_2\neq 0$ and $|\lambda^1_2| \neq |\lambda^1_1|$.  If we now substitute the value of $\lambda^2_1$ in the first equation of~\eqref{12}, a quartic equation in $\lambda^1_1$ arises. Its only valid solutions are $\lambda^1_1 = \epsilon\, \lambda^2_2$, where $\epsilon = \pm1$, thus $\lambda^1_2 = \epsilon\, \lambda^2_1$ and we get~\eqref{L11}.

As a consequence of the values found for $\lambda^2_1$ and $\lambda^2_2$, one should note that~\eqref{de-interes-1} and \eqref{L66} become:
$$
\lambda^3_3 = \epsilon\,\big((\lambda^1_1)^2-(\lambda^1_2)^2\big),\qquad
\lambda^4_4 = \epsilon\,\lambda^5_5=\lambda^1_1\lambda^3_3, \qquad
\lambda^4_5 = \epsilon\,\lambda^5_4= \lambda^1_2\lambda^3_3,$$
$$
\lambda^6_6=\lambda^3_3\,\big( (\lambda^1_1)^2 + (\lambda^1_2)^2 \big), \qquad
\lambda^7_6 = 2\,\epsilon\,\lambda^1_1\lambda^1_2 \lambda^3_3.
$$
Then, one has
$$0=[d(F(e'^{\,8}))-F(de'^{\,8})]_{26}=-\big(\lambda^1_2\lambda^6_6 + \lambda^2_2\lambda^7_6\big)=
-\lambda^1_2\,\lambda^3_3\,\big( 3(\lambda^1_1)^2 + (\lambda^1_2)^2 \big).$$
Since $\lambda^3_3\neq 0$, we are forced to take
$\lambda^1_2 = \lambda^2_1 = 0.$
Therefore, from $[d(F(e'^{\,6}))-F(de'^{\,6})]_{13} =0$ and $[d(F(e'^7))-F(de'^7)]_{23} =0$, we solve
$$\lambda^6_4 = \lambda^1_1\lambda^4_3,\qquad  \lambda^7_5 = \epsilon\,\lambda^1_1\lambda^4_3,$$
and then we get
\begin{equation*}
\begin{split}
0 &=[d(F(e'^{\,8}))-F(de'^{\,8})]_{14} = \lambda^8_6-\lambda^1_1\lambda^6_4
	=\lambda^8_6-(\lambda^1_1)^2\lambda^4_3,\\
0 &=[d(F(e'^{\,8}))-F(de'^{\,8})]_{25} = \lambda^8_6-\lambda^2_2\,(\lambda^5_5+\lambda^7_5)
	=\lambda^8_6-(\lambda^1_1)^2\,(\lambda^3_3+\lambda^4_3).
\end{split}
\end{equation*}
If we solve $\lambda^8_6$ from the first equation above and then replace it in the second one, it suffices to
recall that $\lambda^3_3=\epsilon\,(\lambda^1_1)^2$ to conclude $\lambda^1_1=0$. However, this is not
possible.
\end{proof}

\section{Abelian $J$-invariant ideals}\label{ideales-abelianos}

\noindent As an application of the previous results, in this section we study the existence of non-trivial abelian $J$-invariant ideals. 

\smallskip
Let $\frg$ be a $2n$-dimensional nilpotent Lie algebra endowed with a complex structure $J$. It is clear that, if $J$ is quasi-nilpotent, then $\fra_1(J)\subseteq Z(\frg)$ is a non-trivial abelian $J$-invariant ideal in $\frg$ (see Definition~\ref{tipos_J}). Furthermore, it is proved in \cite[Proposition~1]{LU-procBulgaria} that every $(\frg,J)$, with $J$ either quasi-nilpotent or SnN, has a non trivial $J$-invariant abelian ideal for $n \leq 3$. In the next result, we complete the classification up to eight dimensions.
%
It is worthy to remark that, as a consequence, there are infinitely many (non-isomorphic) 8-dimensional nilpotent Lie algebras $\frg$ for which the only 
abelian $J$-invariant ideal is the trivial one. Moreover, in these cases, this happens for every $J$ defined on $\frg$.

\begin{theorem}\label{th-clasif-ideales}
Let $\frg$ be an NLA of dimension $\leq 8$ endowed with a complex structure $J$. Then, $\frg$ has a non-trivial abelian $J$-invariant ideal if and only if $\frg$ is not isomorphic to the Lie algebras $\frg_9^0$, $\frg_9^1$, or $\frg_{11}^{\alpha,\beta}$.
\end{theorem}

\begin{proof}
It suffices to prove the result for the case of an $8$-dimensional NLA $\frg$ endowed with an SnN complex structure $J$. In this case, note
that $J$ belongs to either Family I or Family II. Hence, up to equivalence, we are reduced to the complex structures equations obtained in Theorem~\ref{main-theorem}. 
Let us denote $\{ Z_j,\bar{Z}_j \}_{j=1}^4$ the complex basis 
dual to $\{ \omega^j,\omega^{\bar j} \}_{j=1}^4$, and consider the real basis 
$\{ X_j, Y_j \}_{j=1}^4$ for $\frg$ given by $X_j=Z_j+\bar{Z}_j$ and $Y_j=JX_j=i\,(Z_j-\bar{Z}_j)$. 

From the structure equations \eqref{FI-SnN} in Family I we easily get that 
$\frf=\langle X_3, X_4, Y_3, Y_4 \rangle$ is a non-trivial $J$-invariant ideal in $\frg$ which is abelian. Similarly, the structure equations \eqref{FII-SnN} in Family II for $\mu=0$ imply that 
$\frf=\langle X_2,X_3, X_4,Y_2, Y_3, Y_4 \rangle$ is a non-trivial abelian $J$-invariant ideal in $\frg$. 

Hence, from now on, we suppose that the complex structure $J$ belongs to Family II with $\mu=1$. 
Again, according to Theorem~\ref{main-theorem}, we have to study the complex structures in the following three particular cases: $(1,1,0,a,b)$, $(0,1,0,0,0)$ and $(0,1,0,1,0)$. 
Next we prove that the only abelian $J$-invariant ideal in $\frg$ is the trivial one.

From the structure equations \eqref{FII-SnN} for $\mu=1$, a direct calculation gives the following (non-zero) brackets for the basis $\{ X_j, Y_j \}_{j=1}^4$:
\begin{equation}\label{brackets}
	\begin{split}
		[X_1,X_2]&=-3\varepsilon\,X_3-2b\,Y_4,\\
		[X_1,X_3]&=-2\,Y_4,\\
		[X_1,X_4]&=-2\,X_2,\\
		[X_1,Y_1]&=2a\,Y_3,\\
	\end{split}
	\qquad
	\begin{split}
		[X_1,Y_2]&=-\varepsilon\,Y_3,\\
		[X_2,X_4]&=-2\,Y_3,\\
		[X_2,Y_1]&=\varepsilon\,Y_3,\\
		[X_2,Y_2]&=-2\,Y_4,\\
	\end{split}
	\qquad\quad
	\begin{split}
		[X_4,Y_1]&=2\,Y_2,\\
		[X_4,Y_2]&=-2\,X_3,\\
		[Y_1,Y_2]&=-\varepsilon\,X_3-2b\,Y_4,\\
		[Y_1,Y_3]&=-2\,Y_4.
	\end{split}
\end{equation}

Let $\frf$ be an abelian $J$-invariant ideal in the Lie algebra $\frg$. Hence, any $U\in\frf$ satisfies in particular that $JU\in\frf$ and $[U,JU]=0$. Let us write $U$ in terms of the real basis above:
\begin{equation*}
U=\sum_{k=1}^4 \big(c_k\,X_k + d_k\,Y_k\big), \text{ where }c_k, d_k\in\mathbb R.
\end{equation*} 
By a direct calculation using~\eqref{brackets}, we get
\begin{equation*}
\begin{array}{rl}
[U,JU]=&-2(c_4d_1-c_1d_4)X_2
-2\left((c_2c_4+d_2d_4)-2\varepsilon(c_1d_2-c_2d_1)\right)X_3\\[5pt]
&+2(c_1c_4+d_1d_4)Y_2+2\left(a(c_1^2+d_1^2)+(c_2d_4-c_4d_2)\right)Y_3\\[5pt]
&-2\left(2\,(c_3d_1-c_1d_3)
-2b(c_1d_2-c_2d_1)+(c_2^2+d_2^2)\right)Y_4.
\end{array}
\end{equation*}
Hence, the condition $[U,JU]=0$ implies that the coefficients of $X_2$ and $Y_2$ in the expression above are zero, so in particular we get  
$c_4(c_1^2+d_1^2)=0=d_4(c_1^2+d_1^2)$. 

If we suppose $c_1=d_1=0$, then the vanishing of the coefficient of $Y_4$ in $[U,JU]$ implies $c_2=d_2=0$, so  
$\frf \subset \langle X_3,X_4,Y_3,Y_4 \rangle$. From \eqref{brackets} and the fact that $\frf$ is an ideal,  we have  
$[U,X_1]=2c_4X_2+2c_3Y_4 \in \frf$, which implies $c_4=0$. 
Similarly, $[JU,X_1]=-2d_4X_2-2d_3Y_4\in \frf$ implies $d_4=0$. 
Thus, $\frf \subset \langle X_3,Y_3 \rangle$ and the same argument gives $c_3=d_3=0$, so $U=0$ and the ideal $\frf$ is zero. 

If we now let $c_1^2+d_1^2\neq0$, then $c_4=d_4=0$ and 
$\frf \subset \langle X_1,X_2,X_3,Y_1,Y_2,Y_3 \rangle$. 
From \eqref{brackets} and the fact that $\frf$ is an ideal,  we have  
$[U,X_3]=-2c_1Y_4$ and $[U,Y_3]=-2d_1Y_4 \in \frf$, which implies $c_1=d_1=0$. 
However, this contradicts the hypothesis $c_1^2+d_1^2\neq0$.
\end{proof}

\section*{Acknowledgments}
\noindent
This work has been partially supported by the projects PID2020-115652GB-I00
(AEI/FEDER, UE),
and E22-17R ``Algebra y Geometr\'{\i}a'' (Gobierno de Arag\'on/FEDER).


\end{document}